\PassOptionsToPackage{compatV3}{fancyhdr}
\documentclass{article}

\def\isarxiv{1} 
\ifdefined\isarxiv
  \usepackage[letterpaper,textwidth=5.5in,textheight=9in]{geometry}
  \usepackage{times}
  \usepackage[colorlinks,citecolor=blue,linkcolor=blue,urlcolor=blue]{hyperref}
  \makeatletter
  \def\@maketitle{\vbox{\hsize\textwidth\centering
    {\LARGE\sc \@title\par}
    \vskip 0.3in minus 0.1in
    {\def\And{\end{tabular}\hfil\linebreak[0]\hfil
              \begin{tabular}[t]{c}\bf\rule{\z@}{24pt}\ignorespaces}
     \def\AND{\end{tabular}\hfil\linebreak[4]\hfil
              \begin{tabular}[t]{c}\bf\rule{\z@}{24pt}\ignorespaces}
     \begin{tabular}[t]{c}\bf\rule{\z@}{24pt}\@author\end{tabular}\par}
    \vskip 0.3in minus 0.1in}}
  \makeatother
  \usepackage{fancyhdr}
  \fancypagestyle{plain}{
    \fancyhf{}
    \fancyhead[L]{Preprint}
    \fancyfoot[C]{\thepage}
    }
\else
  
\fi
\usepackage{graphicx}

\usepackage{amsmath,amssymb}
\usepackage{natbib}
\usepackage{bm}
\usepackage{dsfont}
\usepackage{amsthm}
\usepackage{makecell}
\usepackage{algorithm}
\usepackage{algorithmic}
\usepackage{xcolor}
\usepackage{booktabs}
\usepackage{enumitem}

\usepackage{multirow}
\usepackage{makecell}
\usepackage{amsmath,amssymb,amsthm}
\usepackage{mathtools}
\usepackage{svg}
\usepackage[T1]{fontenc}

\usepackage{caption}
\newcommand{\M}{\mathcal{M}}
\newcommand{\X}{\mathcal{X}}
\newcommand{\Sopt}{\mathcal{S}}
\newcommand{\I}{\mathcal{I}}
\newcommand{\D}{\mathcal{D}}
\newcommand{\R}{\mathbb{R}}

\newcommand{\1}{\mathds{1}}
\newcommand{\E}{\mathbb{E}}
\renewcommand{\L}{\mathcal{L}}
\newcommand{\frameworkName}{TACIT }

\DeclareMathOperator*{\argmin}{arg\,min}

\newcommand{\e}{\varepsilon}

\newcommand{\blue}[1]{{\color{blue} #1}}
\definecolor{darkgreen}{RGB}{0,150,0}
\newcommand{\green}[1]{{\color{darkgreen} #1}}

\theoremstyle{plain}
\newtheorem{theorem}{Theorem}

\newtheorem{definition}{Definition}
\newtheorem{example}{Example}
\usepackage{fvextra}
\usepackage[most]{tcolorbox}
\newtcolorbox{promptbox}[1]{enhanced jigsaw, breakable, toprule at break=0pt, bottomrule at break=0pt, colback=gray!4, colframe=gray!50, boxrule=0.4pt, arc=1pt, left=4pt, right=4pt, fonttitle=\small\bfseries, coltitle=black, colbacktitle=gray!15, title=#1}
\fvset{fontsize=\scriptsize, breaklines=true, breakanywhere=true}

\title{TACIT: Optimization Models that Learn from Their Mistakes}

\author{
  Maxime Bouscary$^{1}$\thanks{This work was performed during an internship at Microsoft Research.} \quad Marco Molinaro$^{2}$ \quad Sirui Li$^{2}$ \\
  \textbf{Saurabh Amin$^{1}$ \quad Ishai Menache$^{2}$ \quad Konstantina Mellou$^{2}$} \\
  $^{1}$Massachusetts Institute of Technology, Cambridge, MA, USA \\
  $^{2}$Microsoft Research, Redmond, WA, USA \\
  {\small\texttt{\{mbscry,amins\}@mit.edu}}\\
  {\small\texttt{\{mmolinaro,siruili,ishai,kmellou\}@microsoft.com}}
}

\begin{document}

\maketitle

\begin{abstract}
Real-world optimization problems are difficult to model accurately because many objectives and constraints reside in domain experts' tacit knowledge, making them hard to formalize. As a result, optimization models often contain miscalibrated objectives, missing constraints, or omitted decision variables, leading to solutions that fail to reflect operational realities. We address this challenge by automatically repairing misspecified formulations using historical data consisting of past solutions and subsequent user overrides. Traditional approaches such as inverse optimization and constraint learning tend to overfit sparse data and produce complex formulations. Our central idea is to combine the reasoning capabilities and prior knowledge of LLMs with the formal grounding provided by optimization. We realize this idea through two complementary paradigms. Top-down, an LLM proposes structural repairs, including new constraints and variables, whose numerical parameters are calibrated and validated through optimization. Bottom-up, optimization infers cuts from observed decisions, which the LLM contextualizes into interpretable, generalizable modeling constraints. We evaluate our approach on 38 misspecification scenarios spanning nine classes of optimization problems, several drawn from real-world applications, and show that \frameworkName can repair 78.9\% of them (vs.~60.5\% for the best baseline).
\end{abstract}

\section{Introduction}

Mathematical optimization is a powerful tool for decision-making, enabling organizations to make complex decisions at a scale and speed that would be difficult to achieve manually. As a result, it plays a central role in automating decision processes across applications such as supply chain and logistics, scheduling, resource allocation, and infrastructure planning. However, an optimization system is only as good as the formulation on which it is built. Constructing such a formulation requires translating a real-world decision problem into decision variables, constraints, and objectives that faithfully capture how decisions should be made in practice. 

This translation is inherently challenging. It typically requires close collaboration between domain experts, who understand the operational requirements and preferences governing decisions, and optimization experts, who encode this knowledge into a mathematical model. Much of this knowledge, however, is tacit and may not be fully articulated during model development. Consequently, optimization formulations can contain misaligned objectives, missing constraints, or omitted variables, producing solutions that fail to reflect operational realities \citep{baxi2026online}.

Large Language Models (LLMs) are increasingly being used to automate parts of the formulation process by translating natural-language problem descriptions into mathematical optimization models~\citep{OptiMUS,astorga2024autoformulation,zhang2024solving}. While recent advances have substantially improved formulation accuracy, they do not eliminate the fundamental challenge of specifying the right decision problem. Requirements absent from the problem description cannot readily be reflected in the resulting formulation, and LLMs may introduce additional formulation errors of their own. Thus, LLM-generated formulations may also be incomplete or misaligned with the decisions users intend to make.

When misaligned optimization systems are deployed in practice, deficiencies in the underlying formulation become visible through their decisions. Recommendations that violate implicit operational requirements or fail to capture user preferences may be rejected, manually overridden, or repeatedly adjusted through tedious processes. Addressing such failures thus often requires costly iterations between domain and optimization experts to diagnose the source of the mismatch and revise the formulation. Yet, these interactions also generate valuable information: rejected recommendations and user corrections 
provide signals about which aspects of the formulation are incomplete or misaligned. This motivates our central question:
\begin{center}
\textit{Can we recover the formulation that governs users' decisions from their corrections, as an explicit and interpretable repair of the deployed one?}
\end{center}

To illustrate the difficulty, consider a toy example with a manufacturer choosing the number of motors ($x$) to produce and machines ($y$) to assemble, where each machine requires one motor. Suppose the formulation mistakenly omits this dependency and captures only production capacity and handling constraints: 
$\max \left\{2x+5y : 2x+3y\le 10,\; x+y\le 4,\; x,y\in\mathbb{Z}_{\ge 0}\right\}$.
Its optimal solution, $(0,3)$, assembles three machines without producing any motors. The user corrects it to $(2,2)$, which is optimal under the intended constraint $x \geq y$. However, this correction does not uniquely reveal the missing constraint: many constraints could make $(2,2)$ optimal, 
while a similar result could be achieved by just changing the objective coefficients. Recovering $x \geq y$ thus requires understanding the semantic relationship between the variables, highlighting the difficulty of repairing optimization models from solution feedback alone.

In practice, users typically provide multiple feedback points as they repeatedly solve instances of the same underlying formulation, but this challenge grows significantly for real-world formulations with a substantial number of variables and constraints, often spanning distinct semantic families and different aspects of the modeled system. More fundamentally, the formulation itself may be incomplete: entire variables may be missing, including auxiliary variables used to represent nonlinear objectives or logical relationships (see Example~\ref{ex:VRP}). Repair may therefore require recovering not only constraints, but missing pieces of the modeling structure itself.

This toy example exposes three broader challenges that make formulation repair fundamentally hard:

\vspace{-0.15cm}

\begin{itemize}[leftmargin=8pt]
    \item \textbf{Indirect feedback.} In practice, users such as supply chain planners or business operators typically focus on correcting the proposed solution rather than diagnosing the underlying model error. 

     \item \textbf{Non-identifiability.} A corrected solution rarely pins down the intended model: the same feedback may be explained by different constraints, objectives, or structural changes, as in the example above. With few corrections, local fixes can explain without capturing the underlying rule.
    
    \item \textbf{Interpretability.} A repair should not merely reproduce observed decisions through an unnatural or unnecessarily complex formulation: it should have a clear real-world interpretation that users can inspect and validate. Recent work highlights the challenges practitioners face in interpreting and modifying optimization models, as well as the broader importance of aligning algorithmic systems with human judgment and organizational objectives~\citep{chen2026optichat,caro2026humanalgorithm}.
\end{itemize}

\vspace{-0.15cm}

\textbf{Our contribution.} We introduce \textsc{TACIT} (Templates And Cuts for Interpretable Transformations), an end-to-end framework for automatically repairing optimization formulations from feedback. We cast repair as a search over modifications of the original formulation: adding or changing constraints, variables, and objective terms. This space is vast and underdetermined: many formulations can be consistent with any finite set of corrections, but few are meaningful to a modeler. Our central idea is to combine the reasoning capabilities and prior knowledge of LLMs with the formal grounding provided by optimization. 
We realize this idea through two complementary paradigms that run within a single search loop. \emph{Bottom-up}, optimization derives cuts that separate rejected solutions from accepted ones, and an LLM agent interprets, refines, and generalizes these cuts into meaningful modeling rules. \emph{Top-down}, an LLM agent identifies a potentially missing modeling component and expresses it as a template, possibly introducing new variables, constraints, or unspecified parameters, which optimization then specifies from the feedback.

This hybrid architecture provides several advantages:
\vspace{-0.15cm}

\begin{enumerate}[leftmargin=12pt]
    \item \textbf{Broad coverage of misspecifications.} Our framework supports a substantially broader class of repairs than local changes to coefficients or individual constraints, recovering missing, extraneous, or miscalibrated components in the objective, constraints, and variables. A single repair can introduce auxiliary variables, add or remove constraints, and modify the objective jointly.

    \item \textbf{Domain generality.} The framework is not tailored to a specific application domain (e.g., transportation, scheduling) and applies across structurally diverse optimization models.

     \item \textbf{No human in the loop.} The framework operates from sparse, offline signals, namely user-corrected solutions, without querying users or requiring them to diagnose the model error.

    \item \textbf{Interpretability and generalization.} By searching over semantically meaningful modeling changes, the framework promotes interpretability while leveraging limited feedback to recover modeling rules that generalize beyond observed data.
\end{enumerate}

\vspace{-0.15cm}

Finally, we construct a comprehensive benchmark of 38 misspecification scenarios spanning nine classes of optimization problems, including several real-world applications, with errors involving constraints, variables, parameters, and objective components. 
Across these settings, \textsc{TACIT} can recover the ground-truth formulation, up to decision equivalence, in 78.9\% of the scenarios from limited user feedback, against 60.5\% for the best baseline. 
The recovered formulations generalize to unseen instances, and our ablations show that combining LLM reasoning with optimization is critical. At the same time, both top-down and bottom-up inference improve repair performance independently, but their combination achieves the best overall performance. Overall, our experiments show that sparse solution feedback can be sufficient to repair complex formulation misspecifications without requiring users to diagnose the formulation themselves.

\section{Related Work} \label{sec:related}

\textbf{Data-driven model customization and preference modeling.} The closest works to ours~\citep{hewitt2020data,bayani2024learning,abdellaoui2026implicit} correct optimization models from prescribed and implemented solutions by learning a function that models directly how users modify the solutions. These approaches range from affine mappings~\citep{hewitt2020data} to regression and decision trees~\citep{bayani2024learning}, and richer predictors for routing and scheduling~\citep{abdellaoui2026implicit}. Related work learns user preferences within an optimization problem~\citep{maragno2025constraint}; see also~\citet{fajemisin2021survey}. There, an ML model modeling user behavior or preference is embedded within the optimization formulation and is trained based on the optimization outcomes. While these data-driven customization methods can substantially improve decisions, the learned behavioral or preference models are encoded within the formulation using additional variables and constraints that do not possess real-world meaning. In contrast, we recover an \emph{explicit, corrected formulation} that can be inspected and validated by humans. Moreover, some complex formulation errors cannot be corrected by such learned-function embeddings, as shown in Theorem~\ref{thm:embed} (Appendix~\ref{app:embed}); our framework can recover the appropriate corrections even in these complex cases. 

\textbf{Inverse optimization and learning optimization models.} Inverse optimization estimates unknown model parameters from observed decisions, including under multiple or noisy observations and imperfect information~\citep{chan2025inverse,moghaddass2021multiple,aswani2018noisy,mohajerin2018imperfect,bulut2021complexity}. Recent methods also jointly learn objectives and feasible regions of LPs or MILPs from contextual examples~\citep{tan2020learning,kumar2021learning,ren2025feasible,kitaoka2025inverse}. This literature motivates our numerical calibration components and highlights the identifiability challenges of decision-only supervision. Moreover, our setting additionally requires recovering the \emph{semantic and structural content} of a misspecified formulation. In addition, these methods cannot naturally capture more complex structural corrections that require introducing new variables and coordinating them with new constraints or objective terms. Our framework addresses both limitations by using LLMs to inject semantic knowledge into the repair process and to propose structural changes involving previously absent variables.

\textbf{Language models for optimization modeling and repair.} LLMs have been used to formulate optimization models from natural-language specifications and diagnose or repair them using solver feedback and domain-specific checks~\citep{OptiMUS,chainofexperts2024,astorga2024autoformulation,zhang2024solving,thind2025optimai,llmopt2025,chen2024diagnosing,ao2026optirepair,ao2026solver}; see~\citet{xiao2025survey} for a recent survey. This literature is complementary to our setting: autoformulation methods assume intended requirements are expressed in the specification, while diagnostic repair relies on explicit specifications, infeasibility, or predefined checks. We instead address requirements that are \emph{missing from the formulation and its specification}, inferring them from downstream behavioral evidence and combining LLM semantic priors with optimization-based validation.

\section{Problem definition}
\label{sec:problem}

\paragraph{Core concepts.}
An instance is described by a data vector $d$ (e.g., demands, capacities, availabilities), whose dimension may vary across instances. We consider mixed-integer linear programs (MILPs), with linear programs (LPs) as a special case. A \emph{formulation} is a pair $\M = (c, \X)$ mapping instance data $d$ to an optimization problem $\M(d)$: $\min c(d)^\top y$ s.t. $y \in \X(d)$,
where $c(d)$ is the vector of objective coefficients and $\X(d)$ the feasibility set, defined by linear constraints and, when applicable, integrality requirements. We use $\Sopt(d; \M)$ to denote its set of optimal solutions. 
An \emph{initial formulation} $\M_0=(c_0,\X_0)$ is given but may not fully capture the underlying decision problem. There is an unknown \emph{target formulation} $\M^\star = (c^\star, \X^\star)$ that is the ground truth that governs the decisions actually taken. The two may differ in the objective (missing terms or incorrect weights), constraints (missing, miscalibrated, or extraneous constraints), variables (e.g., auxiliary variables present only in $\M^\star$), or any combination thereof. 

\begin{example}[Vehicle routing problem (VRP)] \label{ex:VRP} A company operates a fleet $\mathcal L$ of vehicles and must serve a set of customers $C$, with $d_{ij}$ denoting the distance between customers $i$ and $j$. Let variable $x_{ij\ell}$ indicate whether vehicle $\ell$ visits customer $j$ immediately after customer $i$. The company seeks to minimize total distance plus a penalty for routes that exceed $D$ miles. In this example, the initial formulation $\mathcal M_0$ omits this penalty, which appears in the target formulation $\mathcal M^\star$ (in blue):
\vspace{-0.6cm}
\begin{center}
\scalebox{0.85}{\begin{minipage}[t]{0.45\textwidth} \centering \[  \begin{aligned} \mathcal M_0:\quad \min_x\quad & \sum_{\ell\in\mathcal L} \sum_{i,j\in C} d_{ij}x_{ij\ell}\\ \text{s.t.}\quad & x\in\mathcal X_{\mathrm{VRP}}. \end{aligned}  \] \end{minipage}  
\hspace{-0.5cm}
\rule[-3.2cm]{0.4pt}{2.9cm}
\hspace{0.3cm}
\begin{minipage}[t]{0.52\textwidth} \centering \[  \begin{aligned} \mathcal M^\star:\quad \min_{x,z}\quad & \sum_{\ell\in\mathcal L} \sum_{i,j\in C} d_{ij}x_{ij\ell} + \blue{\lambda\sum_{\ell\in\mathcal L}z_\ell}\\ \text{s.t.}\quad & x\in\mathcal X_{\mathrm{VRP}},\\ & \blue{\sum_{i,j\in C}d_{ij}x_{ij\ell} \le D+Mz_\ell,} && \blue{\forall \ell\in\mathcal L}\\ & \blue{z_\ell\in\{0,1\},} && \blue{\forall \ell\in\mathcal L,} \end{aligned} \] 
\end{minipage} }
\end{center}
$\mathcal X_{\mathrm{VRP}}$ denotes the standard VRP constraints. Repairing $\mathcal M_0$ requires coordinated changes: adding auxiliary variables, linking constraints, a new objective term, and estimating its coefficient $\lambda$.
\end{example}

\vspace{-0.15cm}

We refer to the variables of the initial formulation $\M_0$ as \emph{primary variables}: they encode the actionable decisions made and reported by practitioners. In Example~\ref{ex:VRP}, $x_{ij\ell}$ are primary, while the auxiliary variables $z_\ell$ introduced by $\M^\star$ are not.  Any additional variables introduced by $\M^\star$ are called \emph{auxiliary variables}. Only the primary variables of $\M^\star$ are observed in the user feedback. Let $\Pi_0$ be the projection that restricts a solution to the primary variables. A \emph{feasible lifting} of $x$ under $\M$ is any solution $y\in\X(d)$ satisfying $\Pi_0(y)=x$.

\textbf{Observed data.}
Instances are drawn from an unknown distribution $P$ over triples $(d, x^0, x^\star)$, where $x^0 \in \Sopt(d;\M_0)$ is the \emph{initial solution} and $x^\star \in \Pi_0(\Sopt(d;\M^\star))$ the \emph{target solution}, the projection onto the primary variables of an optimal solution of $\M^\star(d)$. We observe a sample $\D = \left\{ \left(d_i, x_i^{0}, x_i^{\star}\right) : i \in \I \right\} \sim P^{|\I|}$.
Optimality of the targets is a standard assumption of inverse optimization \citep{chan2025inverse, besbes2025contextual}: we use it to describe the problem and our methodology and relax it in some of our experiments. We further posit that whenever $x_i^0 \neq x_i^\star$, the target is strictly preferred, that is, every lifting of $x_i^0$ is infeasible or strictly suboptimal under $\M^\star$. This reflects practice, where the initial solution is the default and practitioners depart from it only when an alternative is strictly better. 

\textbf{Separation loss.}
Our goal is to find a formulation $\M$ that best explains the observed data $\D$. We formalize this  through \emph{separability}: $\M$ \emph{separates} $\D$ if, for every instance $i$ with $x_i^0 \neq x_i^\star$, $x_i^\star \in \Pi_0\!\left(\Sopt(d_i;\M)\right)$ and $x_i^0 \notin \Pi_0\!\left(\Sopt(d_i;\M)\right)$.
That is, the target solution is feasible and optimal under $\M$, whereas the initial solution is not. By the assumptions above, $\M^\star$ 
separates $\D$.
To quantify the degree of separation achieved by a candidate formulation $\M=(c,\X)$, we measure the feasibility and optimality of a solution $x$ on instance $i$ through its relative optimality gap 
  $\gamma_i(x) \triangleq \frac{Z_i(x)-Z_i^{\mathrm{opt}}}{|Z_i^{\mathrm{opt}}|},$
where $Z_i^{\mathrm{opt}} = \min\{ c(d_i)^\top y : y\in\X(d_i)\}$ and $Z_i(x) = \min \{ c(d_i)^\top y : y\in\X(d_i), \Pi_0(y)=x\}$
denote, respectively, the optimal objective value of $\M(d_i)$ and the best objective value among feasible liftings of $x$.\footnote{When $Z_i^{\mathrm{opt}}=0$, we add a small positive constant to the denominator, and if $x$ does not admit a feasible lifting under $\M(d_i)$ we define $\gamma_i(x) = 1$.} By construction, $\gamma_i(x)=0$ if and only if $x\in\Pi_0(\Sopt(d_i;\M))$.
Given tolerances $0 \leq \delta \leq \Delta < 1$, we define the \emph{separation loss}:
\begin{equation} \label{eq:separation_loss}
  \L(\M; \D) \triangleq 1-\frac{1}{|\I|} \sum_{i \in \I} r_i(\M), \;
  r_i(\M) \triangleq
  \begin{cases}
      \1_{\left\{\gamma_i(x_i^\star) \leq \Delta\right\}} \cdot \1_{\left\{\gamma_i(x_i^0) \geq \gamma_i(x_i^\star) + \delta\right\}} & \text{if } x_i^0 \neq x_i^\star\\
      \1_{\left\{\gamma_i(x_i^\star) \leq \Delta\right\}} & \text{otherwise.}
  \end{cases}
\end{equation}
Here, $\Delta$ bounds the relative optimality gap of the observed target solutions, while $\delta$ is the minimum gap difference separating the initial solution from the target. Thus, $r_i(\M)=1$ if $\M$ makes $x_i^\star$ near-optimal within tolerance $\Delta$ and, when $x_i^0\neq x_i^\star$, prefers $x_i^\star$ to $x_i^0$ by at least $\delta$. In this case, we say that $\M$ separates instance $i$. The loss $\L(\M; \D) \in [0,1]$ is thus the fraction of instances that $\M$ fails to separate, and $\L(\M; \D) = 0$ if and only if $\M$ separates $\D$.

We frame the formulation correction problem as finding a formulation minimizing the expected separation loss: $\min_{\M} \; \E_{\D \sim P}\left[\L(\M; \D)\right]$.
Since $P$ is unknown, we minimize $\L(\M;\D)$ on the observed dataset $\D$ and assess generalization on held-out instances. As there can be multiple equivalent formulations differing widely in size and readability, our method additionally aims for a parsimonious transformation of $\M_0$ 
that remains interpretable for practitioners maintaining $\M_0$.

\section{Methodology} \label{sec:methodology}

\begin{figure}[t]
  \centering
  \includegraphics[trim={0cm 0.1cm 0cm 0cm}, clip,width=0.98\columnwidth]{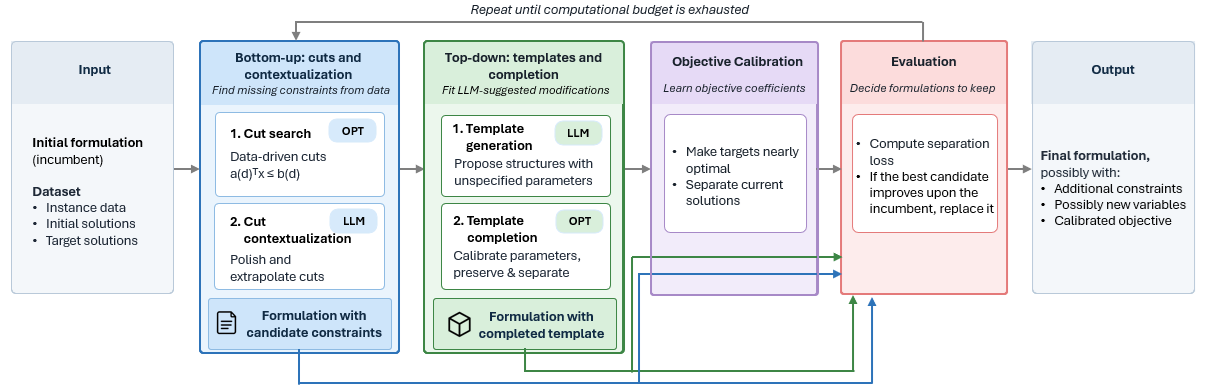}
  \caption{\textsc{TACIT} framework overview.} 
  \label{fig:framework_diagram_K} 
\end{figure}

Our framework is illustrated in Figure~\ref{fig:framework_diagram_K}. We maintain an incumbent formulation $\M=(c,\X)$, initialized by removing from $\M_0$ any constraints that render a target solution infeasible. We then iteratively generate, evaluate, and incorporate candidate transformations.
Candidates are generated via two complementary paradigms that combine optimization and LLM agents in opposite directions. In the \emph{bottom-up} paradigm, optimization finds linear inequalities separating solutions of the incumbent formulation from the targets, which an LLM agent contextualizes and extrapolates into meaningful constraints or constraint families. In the \emph{top-down} paradigm, an LLM agent proposes a transformation template whose learnable terms are calibrated through optimization. Candidates are evaluated using the separation loss \eqref{eq:separation_loss}, before and after objective calibration, and an improving candidate becomes the new incumbent. These steps are applied sequentially and repeated until the computational budget is exhausted; Algorithm~\ref{alg:full_framework} gives the full procedure.

Finding an improved formulation presents several challenges: (i) handling multiple and general formulation errors; (ii) navigating the vast space of possible formulations; (iii) constructing inequalities and objectives that apply across different instance sizes; (iv) producing interpretable transformations; and (v) generalizing to unseen instances under limited data.
We next detail the formulation transformations in our system and how they address these challenges.

\subsection{Bottom-up: cuts and contextualization} \label{sec:bottom-up}

\textbf{Cut search.}
A \emph{cut} for formulation $\M$ is a linear inequality satisfied by every target solution $x_i^\star$ but excludes at least one of the projected optimal solutions $\hat{x}_i$ of $\M$. We seek to add cuts to the current incumbent formulation $\M$ that is not yet fully correct. 
Challenges (ii) and (iii) immediately appear. To be deployed across instances, a cut has to be a function of the instance data $d$, which leads to a large search space. We restrict this space by considering cuts $a(d)^\top x \le b(d)$ where the components $a(d)$ and $b(d)$ are affine functions of the data. Moreover, to handle varying instance sizes (e.g., VRP with different number of customers), we employ \emph{family-level weights} to define the cuts: each coefficient has the form $a(d)_k = \alpha_{t(k)}^\top \phi_k(d)$, where $\phi_k(d)$ is the data associated with variable $k$ and $\alpha_{t(k)}$ is the weight vector (to be estimated) for its the family $t(k)$ of this variable; similarly $b(d) = \beta^\top \phi_0(d)$, where $\phi_0(d)$ represents data not associated to a particular variable family and $\beta$ are the weights to be estimated. More details about this concept can be found in Appendix~\ref{app:cuts}.

We infer $\alpha$ and $\beta$ through a MILP (see Appendix~\ref{app:cuts}) that maximizes the cut quality score  
\begin{align}
\max_{\alpha,\beta,z,\eta} \quad \frac{1}{|\mathcal I|}\sum_{i\in\mathcal I} z_i - \rho\bigl(\|\alpha\|_0 + \|\beta\|_0\bigr) + \e\eta, \end{align}
subject to validity for every $x_i^\star$, where $z_i$ indicates whether $\hat{x}_i$ violates the cut by at least a margin $\eta\ge\underline{\eta}>0$. The score therefore rewards 
cuts that eliminate a large fraction of the solutions from formulation $\M$, favors simple and sparse structure, and, through the small term $\e\eta$ with $0<\e\ll1$, breaks ties in favor of stronger separation.

Sparsity regularization limits overfitting, promoting cut generalization to unseen instances, and also improves interpretability, since meaningful combinatorial constraints often involve only few variables and instance features (Challenges (iv)-(v)).
We sweep the parameter $\rho$, generating cuts from the sparsest nonempty solution to the unregularized one, and evaluate them by separation loss. These candidates are then passed to the \emph{cut contextualization} step for polishing and extrapolation.

\textbf{Cuts contextualization.}  The cut-search procedure identifies sparse numerical structures that explain the observed corrections, but the resulting coefficients need not directly reveal the underlying modeling intent. We therefore invoke an LLM agent to contextualize each discovered cut using the full formulation~$\mathcal{M}$. The agent first \emph{polishes} the cut by proposing semantically plausible transformations, such as rounding coefficients, removing negligible terms, and introducing contextually relevant terms. For example, the search may return the cut $0.98x_i + 1.03x_j - 1.01y_k \leq 0.97$, while the underlying combinatorial structure is more naturally expressed as $x_i +x_j- y_k\leq 1$. Thus, the search identifies the relevant geometry, while the LLM uses the formulation context to recover a cleaner symbolic representation, which is crucial for interpretability. Importantly, the LLM does not unilaterally select the repair: it proposes a collection of candidate transformations, which are evaluated using the separation loss \eqref{eq:separation_loss}, and the lowest-loss variant is retained.

The agent then \emph{extrapolates} the accumulated cuts into indexed constraint families. For instance, $x_{3,7}\leq y_3$ may suggest the family $x_{ij}\leq y_i,\ \forall i,j$. Such families are ubiquitous (e.g., Example~\ref{ex:VRP}) and can encode many individual constraints. 
With limited data (Challenge~(v)), observed corrections may reveal only a few members of the underlying family, making extrapolation important for recovering its broader structure. Table~\ref{tab:cut-examples} provides cut contextualization examples from our experiments.

\subsection{Top-down: templates and calibration} \label{sec:top-down}

Some target formulations introduce additional variables and therefore cannot be recovered through cuts alone (Theorem~\ref{thm:embed}). The \emph{top-down paradigm} addresses this limitation by having an LLM agent propose a \emph{template} for the transformation, potentially introducing auxiliary variables, modifying variable domains, and adding constraint families with unspecified parameters. Optimization then calibrates these parameters to preserve target solutions while maximally separating unwanted ones.

\textbf{Template generation.} Rather than asking the LLM to produce a fully specified repair, we use it to propose a \emph{structural hypothesis}  for how the formulation should change. The agent receives the incumbent formulation $\M$, constraints removed during preprocessing, and samples of mismatched pairs $(\hat x_i,x_i^\star)$, where $\hat x_i$ is a projected optimal solution of $\M$, i.e., $\hat x_i=\Pi_0(\hat y_i)$ for $\hat y_i\in\Sopt(d_i;\M)$. It proposes a template $\tau$ specifying \emph{what} should change, while leaving its learnable parameters unspecified. This separation is important: the agent identifies a semantically meaningful modeling mechanism, while the subsequent calibration step determines whether and how it can be instantiated without excluding the observed targets. Table~\ref{tab:template-examples} provides template examples from our experiments.
Formally, a template $\tau$ specifies (i) new auxiliary variables $z$ with domain $Z_\tau(d)$; (ii) constraints $g_j(y,z,w;d)\leq 0$, $j\in J_\tau(d)$, parameterized by weights $w$ to be calibrated; and (iii) a possibly modified domain $\mathcal{Y}_\tau(d)$ for the existing variables $y$. Dependence on $d$ allows templates to generalize across instance sizes. The functions $g_j$ are bi-affine in $(y,z)$ and $w$, yielding the updated feasible set
\begin{equation}
    \X_\tau(d; w) = \left\{ (y, z) : y \in \mathcal{Y}(d), \; z \in Z_\tau(d), \; g_j(y, z, w; d) \le 0, \; \forall j \in J_\tau(d) \right\}.
\end{equation}
A template does not change the objective function and assigns zero objective coefficients to new variables; their contribution is determined separately by the objective-calibration step in Section~\ref{sec:objective-calibration}.

For example, in the VRP setting of Example~\ref{ex:VRP}, the missing long-route penalty cannot be recovered through cuts over the existing variables alone. A template can introduce binary variables $z_\ell$, $\ell\in\mathcal L$, indicating long routes and constraints $g_\ell(x,z,w;d)\leq0$, $\ell\in\mathcal L$, linking them to route lengths. Calibration may instantiate these as $\sum_{i,j\in C} d_{ij}x_{ij\ell}\leq D+Mz_\ell$. Thus, the LLM agent identifies the missing modeling structure, while optimization determines its numerical parameters, enabling repairs involving multiple connected modeling changes (Challenge~(i)).

\textbf{Template completion.} We calibrate the template weights $w$ to preserve all observed targets while maximally separating the optimal solutions of the current formulation $\M$. Preservation requires each $x_i^\star$ to admit a lifting $y_i^\star$ into the non-primary variables of $\M$ and an assignment of the new variables $z_i^\star$ such that $(y_i^\star,z_i^\star)$ belongs to the new feasible set $\X_\tau(d_i; w)$. 
We therefore solve an optimization problem with max-min structure to jointly calibrate the template parameters $w$ against the best completions $\hat z_i$ of the new variables. We solve this problem via constraint generation and discard any template where no weights allow every target to be feasibly lifted; see Appendix~\ref{app:template_completion}.

\subsection{Objective calibration} \label{sec:objective-calibration} 

Cuts and templates modify the feasible set, but misspecification may also lie in the objective. Moreover, auxiliary variables introduced by templates may matter via their cost. Thus, every candidate formulation $\M=(c,\X)$ undergoes \emph{objective calibration}, keeping $\X$ fixed while learning $c$. Calibration balances two goals: (i) each target $x_i^\star$ should admit a feasible lifting $y_i^\star$ that is nearly optimal under the calibrated objective, allowing some suboptimality to accommodate potentially suboptimal target solutions in practice, (ii) the lifted target $y_i^\star$ should outperform the current model solution $\hat y_i$ by as large a margin as possible. Appendix~\ref{app:obj_calibration} provides the resulting optimization problem.

\section{Numerical Results}

\subsection{Experimental Setup} \label{sec:setup}

\textbf{Dataset.} We introduce a benchmark of 38 scenarios spanning nine classes of LPs and MILPs. The benchmark combines classical optimization problems with realistic applications, including scenarios motivated by resource allocation in hyperscalers. 
The scenarios cover a broad range of formulation errors along two dimensions: the affected component (objective, constraints, variables, domains) and nature of the error (extraneous, missing, miscalibrated, local, global). Many scenarios deliberately combine multiple forms of misspecification. Each scenario specifies an initial formulation $\M_0$ and a target formulation $\M^\star$. Appendix~\ref{app:problems_and_scenarios} gives them in full and Table~\ref{tab:scenario-taxonomy} summarizes the taxonomy. We will release the benchmark for reproducibility and future research on formulation repair.

\textbf{Baselines and protocol.} We compare our framework against (i) the initial formulation, (ii) the data-driven methods of \citet{hewitt2020data,bayani2024learning}, which we denote AOP, and (iii) an LLM-only baseline that generates the complete optimization model as \texttt{gurobipy} code from the initial code and pairs of model and target solutions. For each scenario, every method is run 5 times using $N=128$ training instances and evaluated on a disjoint set of 256 test instances.

\textbf{Metrics.} We evaluate the final incumbent formulation $\M$ on both training and test instances. On the training set, we report the separation loss $\L$ and the \emph{separation rate}, defined as the fraction of runs with $\L=0$. On the test set, we report four metrics. The \emph{feasibility rate} is the fraction of test instances for which the solution produced by $\M$ is feasible under $\M^\star$. The \emph{optimality rate} is the fraction for which that solution is feasible and within $1\%$ of the optimal value under $\M^\star$ in relative gap. The \emph{recovery rate} is the proportion of scenarios for which this optimality criterion holds on \textbf{all 256 test instances}. These metrics are averaged over the 5 runs of each method. Finally, the \emph{recoverable rate} is the proportion of scenarios for which at least one of the 5 runs attains recovery, capturing the fraction of scenarios that are amenable to formulation repair by \textsc{TACIT}.

\subsection{Results} \label{sec:results}

We run \textsc{TACIT} with three LLM backbones (DeepSeek-V4-Flash, Kimi-K2.6, GPT-5.6-sol) for cut contextualization and template generation, using five iterations of the framework per run.  To match this budget, the LLM-only baseline is given five attempts per run, and its best formulation by training separation loss is kept. Additional implementation details are provided in Appendix~\ref{app:implementation_details}.

\textbf{TACIT substantially improves formulation repair across LLM backbones.} Table~\ref{tab:llms_comparison} compares TACIT with LLM-only repair and AOP. TACIT improves every test metric across all backbones. AOP (Adapted Optimization Problem) never recovers a correct formulation: increasing its penalty weight $c$ makes solutions significantly suboptimal before it makes them feasible (Table~\ref{tab:AOP} in the appendix). With DeepSeek-V4-Flash and Kimi-K2.6, TACIT nearly doubles the scenario recovery and recoverable rates compared to LLM-only: $19.5\% \rightarrow 45.3\%$ and $31.6\% \rightarrow 63.2\%$ respectively for DeepSeek-V4-Flash; $30\% \rightarrow 55.3\%$ and $36.8\% \rightarrow 71.1\%$, for Kimi-K2.6. Importantly, TACIT remains beneficial even with the substantially stronger GPT-5.6-sol, where TACIT obtains additional $17.4$ and $18.4$ percentage points, respectively, and achieves a total $64.2\%$ recovery and $78.9\%$ recoverable rate of scenarios. Overall, TACIT improves performance across all reasoning models, with larger gains for weaker models, and is able to lift weaker models to approach the performance of much stronger models. Table~\ref{tab:llms_comparison} further shows that feasibility is recovered reliably across all backbones, suggesting that feasibility is not the primary bottleneck in formulation recovery; the harder challenge is recovering target-optimal behavior.

\begin{table}[t]
\caption{\frameworkName against LLM baselines under different LLM backbone models.}
\vspace{-0.2cm}
\label{tab:llms_comparison}
\scalebox{0.95}{
\centering
\footnotesize
\begin{tabular}{llcccccccc}
\toprule
 & & \multicolumn{4}{c}{Training} & \multicolumn{4}{c}{Test} \\
\cmidrule(lr){3-6} \cmidrule(lr){7-10}
 & Model & $\L$ $\downarrow$ & sep. $\uparrow$ & feas. $\uparrow$ & opt. $\uparrow$ & feas. $\uparrow$ & opt. $\uparrow$ & recovery $\uparrow$ & recoverable $\uparrow$ \\
\midrule
\multirow{2}{*}{\rotatebox[origin=c]{90}{\scriptsize AOP}}
& AOP linear ($c\!=\!\!16$) & 0.94 & 0.0\% & 0.57 & 0.07 & 0.57 & 0.04 & 0.0\% & 0.0\% \\
& AOP tree ($c\!=\!\!16$) & 0.93 & 0.0\% & 0.58 & 0.10 & 0.57 & 0.07 & 0.0\% & 0.0\% \\
\addlinespace
\multirow{3}{*}{\rotatebox[origin=c]{90}{\scriptsize LLM-only}}
 & DeepSeek-V4-F & 0.64 & 18.9\% & 0.57 & 0.37 & 0.57 & 0.37 & 19.5\% & 31.6\% \\
 & Kimi-K2.6 & 0.56 & 29.5\% & 0.66 & 0.46 & 0.66 & 0.47 & 30.0\% & 36.8\% \\
 & GPT-5.6-sol & 0.40 & 47.9\% & 0.74 & 0.59 & 0.75 & 0.59 & 46.8\% & 60.5\% \\
\addlinespace
\multirow{3}{*}{\rotatebox[origin=c]{90}{\scriptsize \frameworkName}}
 & DeepSeek-V4-F & 0.19 & 57.4\% & 0.79 & 0.64 & 0.78 & 0.61 & 45.3\% & 63.2\% \\
 & Kimi-K2.6 & \textbf{0.12} & \textbf{66.8\%} & 0.87 & 0.77 & 0.85 & 0.73 & 55.3\% & 71.1\% \\
 & GPT-5.6-sol & 0.12 & 66.3\% & \textbf{0.89} & \textbf{0.79} & \textbf{0.88} & \textbf{0.76} & \textbf{64.2\%} & \textbf{78.9\%} \\
\addlinespace
\bottomrule
\end{tabular}
}
\end{table}

\textbf{TACIT provides the largest gains on structurally demanding misspecifications.} Figure~\ref{fig:framework_diagram} breaks down test optimality by misspecification type, revealing a clear spectrum of difficulty. Some errors, such as variable-domain misspecifications, can be readily identified from the formulation context. In contrast, LLM-only baselines struggle when repair requires reconstructing structure absent from the initial formulation, such as missing variables or global constraints. TACIT substantially improves performance on these harder categories across backbones. For example, on missing variables, TACIT improves test optimality $0.30 \rightarrow 0.44$ with DeepSeek-V4-Flash and $0.46 \rightarrow 0.74$ with Kimi-K2.6; on global constraints, the improvements are $0.21 \rightarrow 0.47$ and $0.35\rightarrow 0.69$. Similar gains arise for extra constraints. These results suggest that TACIT is most valuable when formulation repair goes beyond recognizing local errors and requires reconstructing missing model structure.

\begin{figure}[t]
  \centering
  \includegraphics[trim={0cm 0.27cm 0cm 0.35cm}, clip, width=0.99\columnwidth]{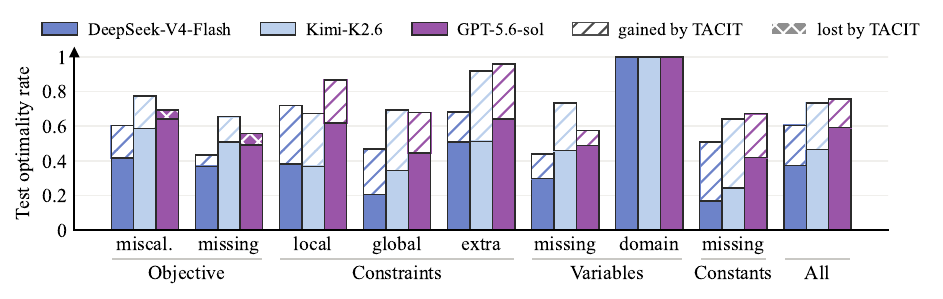}
  \caption{Change in optimality rate induced by \frameworkName under different LLM backbone models.}
  \label{fig:framework_diagram} 
\end{figure}

\textbf{LLM reasoning and optimization play complementary roles.} We ablate \textsc{TACIT} with DeepSeek-V4-Flash in Table~\ref{tab:individual_components}. The top-down pathway is already strong, achieving recovery and recoverability of $41.1\%$ and $57.9\%$, but incorporating the bottom-up pathway further improves them to $45.3\%$ and $63.2\%$. The results also highlight the importance of combining LLM reasoning with optimization: the best pure approach, whether based only on optimization-driven cut generation or on LLM-only templates, achieves recovery and recoverability of $21.1\%$ and $31.6\%$, less than half of the full system. Adding optimization-based calibration to the strongest pure approach, LLM-generated templates, raises recovery $21.1\%\rightarrow41.1\%$ and recoverability $31.6\% \rightarrow 57.9\%$. 
The breakdown in Table~\ref{tab:methods-misspec} helps explain this complementarity: cut-based inference is particularly useful for local constraint errors, whereas templates are stronger when the repair requires new modeling structure, such as auxiliary variables. Together, these results show that TACIT benefits both from combining LLM reasoning with optimization and from combining top-down semantic reasoning with bottom-up evidence from observed decisions.

\begin{table}[t]
\caption{Ablation of TACIT components using DeepSeek-V4-Flash. All variants share preprocessing and objective calibration. \texttt{Template} only uses LLM-guessed constants.}
\vspace{-0.2cm}
\label{tab:individual_components}
\centering
\footnotesize
\scalebox{0.95}{
\begin{tabular}{lcccccccc}
\toprule
 & \multicolumn{4}{c}{Training} & \multicolumn{4}{c}{Test} \\
\cmidrule(lr){2-5} \cmidrule(lr){6-9}
 & $\L$ $\downarrow$ & sep. $\uparrow$ & feas. $\uparrow$ & opt. $\uparrow$ & feas. $\uparrow$ & opt. $\uparrow$ & recovery $\uparrow$ & recoverable $\uparrow$ \\
\midrule
Obj. calibration only & 0.76 & 5.8\% & 0.41 & 0.23 & 0.41 & 0.23 & 5.3\% & 5.3\% \\
Cuts & 0.35 & 33.2\% & 0.57 & 0.33 & 0.54 & 0.27 & 14.7\% & 21.1\% \\
Cuts + contextualize & 0.35 & 33.2\% & 0.59 & 0.35 & 0.57 & 0.30 & 17.9\% & 26.3\% \\
Template & 0.60 & 22.6\% & 0.56 & 0.39 & 0.56 & 0.39 & 21.1\% & 31.6\% \\
Template + complete & 0.31 & 47.9\% & 0.77 & 0.59 & 0.77 & 0.58 & 41.1\% & 57.9\% \\
\frameworkName & \textbf{0.19} & \textbf{57.4\%} & \textbf{0.79} & \textbf{0.64} & \textbf{0.78} & \textbf{0.61} & \textbf{45.3\%} & \textbf{63.2\%} \\
\bottomrule
\end{tabular}
}
\end{table}

\textbf{Performance under limited training data.} Table~\ref{tab:sizes} shows that top-down inference is stable across training sizes, whereas bottom-up improves by 12.6 percentage points from 32 to 256 instances. \frameworkName combines both behaviors, maintaining strong performance across training sizes.

\textbf{Robustness to suboptimal feedback.}
We relax the assumption that observed target decisions are optimal, allowing up to $5\%$ suboptimality. Test feasibility remains unchanged ($0.78$ to $0.79$), while test recovery decreases by 4.8 percentage points. This is a deliberately challenging setting: \frameworkName may never observe a target-optimal solution, yet is evaluated on whether it recovers target-optimal behavior. These results suggest that imperfect feedback still conveys useful structural information, allowing \frameworkName to remain robust even as exact preference recovery becomes harder; see Table~\ref{tab:subopt}.

\textbf{The recovered transformations remain explicit and inspectable.} Beyond predictive performance, TACIT returns modifications to the optimization formulation rather than an opaque correction of its outputs. Section~\ref{sec:final_form_samples} illustrates this across three scenarios. In TP-Scenario2, it recovers the target exactly, adding $x_{ij} \le 24 y_{ij}$ for all $i \in I$, $j \in J$, and $\sum_{i \in I} y_{ij} \le 2$ for all $j \in J$. In the more challenging MCF-Scenario2, it fixes the objective function by introducing auxiliary variables and the cuts $f_{ij} \le u_{ij} y_{ij}$ for all $i \in V, j \in O_i$, plus the valid but redundant $0 \cdot y_{ij} \le f_{ij}$. Results are not always as clean. In GPU-Scenario4, the recovered objective $-1.90\sum_{i\in I} s_i z_i - 0.10\sum_{i\in I} z_i - 2.42\sum_{i\in I}\sum_{j\in J_i} x_{ij}$ is equivalent under the problem constraints to a scaling of the target $-\sum_{i \in I} s_i z_i$ with a small tie-breaking term $-\epsilon \sum_{i\in I} z_i$. 
Table~\ref{tab:cut-examples} further shows how the cut-contextualization agent turns numerical cuts into interpretable constraints.
Across inspected examples, TACIT recovers recognizable modeling constructs, such as linking constraints, combinatorial restrictions, auxiliary variables, and structured objective terms, rather than opaque instance-specific corrections. Even when non-canonical terms are present, the recovered formulations remain directly inspectable.

\section{Conclusion}

Formulation repair is challenging: the source of misspecification is unknown, modeling logic may be absent from the initial formulation, and sparse feedback can be consistent with many possible repairs. We introduced TACIT, an end-to-end framework combining LLM-guided semantic reasoning with optimization-based search and validation to address this challenge. Across diverse formulation errors, it substantially improves recovery over the baselines, especially on structurally demanding misspecifications, while remaining effective with limited and imperfect feedback. Our results demonstrate that user corrections provide sufficient signal to systematically improve optimization models, a step toward models that learn from their mistakes.

\ifdefined\isarxiv

\else

\section*{Reproducibility Statement}
All 38 benchmark scenarios, including the base problems, initial and target formulations, and instance generation procedures, are specified in Appendix~\ref{app:problems_and_scenarios}. The cut-search MILP, template completion, and objective calibration problems are given in Appendices~\ref{app:cuts}, \ref{app:template_completion}, and~\ref{app:obj_calibration}, respectively. Appendix~\ref{app:implementation_details} details the prompts, compute budgets, and hyperparameters used for numerical experiments. The proof of Theorem~\ref{thm:embed} is in Appendix~\ref{app:embed}. The full code will be released upon acceptance.

\section*{AI Use Statement}

Beyond their role in the proposed framework, LLMs were used to polish the writing, search for related work, and help in implementing the benchmark scenarios. All text, references, and scenarios were verified by the authors, who take full responsibility for the manuscript and its artifacts.

\fi

\bibliography{refs.bib}

@article{ao2026optirepair,
  title={Optirepair: Closed-loop diagnosis and repair of supply chain optimization models with LLM agents},
  author={Ao, Ruicheng and Simchi-Levi, David and Wang, Xinshang},
  journal={arXiv preprint arXiv:2602.19439},
  year={2026}
}

@article{chan2025inverse,
  title={Inverse optimization: Theory and applications},
  author={Chan, Timothy CY and Mahmood, Rafid and Zhu, Ian Yihang},
  journal={Operations Research},
  volume={73},
  number={2},
  pages={1046--1074},
  year={2025},
  publisher={INFORMS}
}

@article{hewitt2020data,
  title={Data-driven optimization model customization},
  author={Hewitt, Mike and Frejinger, Emma},
  journal={European Journal of Operational Research},
  volume={287},
  number={2},
  pages={438--451},
  year={2020},
  publisher={Elsevier}
}

@article{ao2026solver,
  title={Solver-in-the-Loop: MDP-Based Benchmarks for Self-Correction and Behavioral Rationality in Operations Research},
  author={Ao, Ruicheng and Simchi-Levi, David and Wang, Xinshang},
  journal={arXiv preprint arXiv:2601.21008},
  year={2026}
}

@article{chen2024diagnosing,
  title={Diagnosing infeasible optimization problems using large language models},
  author={Chen, Hao and Constante-Flores, Gonzalo E and Li, Can},
  journal={INFOR: Information Systems and Operational Research},
  volume={62},
  number={4},
  pages={573--587},
  year={2024},
  publisher={Taylor \& Francis}
}

@article{OptiMUS,
  title={OptiMUS-0.3: Using large language models to model and solve optimization problems at scale},
  author={AhmadiTeshnizi, Ali and Gao, Wenzhi and Brunborg, Herman and Talaei, Shayan and Lawless, Connor and Udell, Madeleine},
  journal={arXiv preprint arXiv:2407.19633},
  year={2024}
}

@article{thind2025optimai,
  title={Optimai: Optimization from natural language using llm-powered ai agents},
  author={Thind, Raghav and Sun, Youran and Liang, Ling and Yang, Haizhao},
  journal={arXiv preprint arXiv:2504.16918},
  year={2025}
}

@article{astorga2024autoformulation,
  title={Autoformulation of mathematical optimization models using llms},
  author={Astorga, Nicol{\'a}s and Liu, Tennison and Xiao, Yuanzhang and van der Schaar, Mihaela},
  journal={arXiv preprint arXiv:2411.01679},
  year={2024}
}

@inproceedings{zhang2024solving,
  title={Solving general natural-language-description optimization problems with large language models},
  author={Zhang, Jihai and Wang, Wei and Guo, Siyan and Wang, Li and Lin, Fangquan and Yang, Cheng and Yin, Wotao},
  booktitle={Proceedings of the 2024 Conference of the North American Chapter of the Association for Computational Linguistics: Human Language Technologies (Volume 6: Industry Track)},
  pages={483--490},
  year={2024}
}

@article{bayani2024learning,
  title={Learning and modeling implicit constraints in optimiza-tion models through decision trees},
  author={Bayani, M and Adulyasak, Y and Rousseau, LM},
  journal={Les Cahiers du GERAD ISSN},
  volume={711},
  pages={2440},
  year={2024}
}

@article{chen2026optichat,
  author  = {Chen, Hao and Constante-Flores, Gonzalo Esteban and
             Mantri, Krishna Sri Ipsit and Kompalli, Sai Madhukiran and
             Ahluwalia, Akshdeep Singh and Li, Can},
  title   = {OptiChat: Bridging Optimization Models and Practitioners
             with Large Language Models},
  journal = {INFORMS Journal on Data Science},
  volume  = {5},
  number  = {3},
  pages   = {199--220},
  year    = {2026},
  doi     = {10.1287/ijds.2025.0074}
}

@article{caro2026humanalgorithm,
  author  = {Caro, Felipe and Colliard, Jean-Edouard and Katok, Elena and
             Ockenfels, Axel and Stier-Moses, Nicolas and Tucker, Catherine and
             Wu, D. J.},
  title   = {Introduction to the Special Issue on the Human-Algorithm Connection},
  journal = {Management Science},
  volume  = {72},
  number  = {1},
  pages   = {1--13},
  year    = {2026},
  doi     = {10.1287/mnsc.2023.intro.v72.n1}
}

@article{abdellaoui2026implicit,
  author  = {Abdellaoui, Abdelhakim and Boufous, Ayoub and El Hallaoui, Issmail and Benabbou, Loubna and Aube, Francois and Amazouz, Mouloud},
  title   = {Learning Implicit Feasibility Constraints for Real-World Routing and Scheduling: Application to Log Transportation},
  journal = {arXiv preprint arXiv:2606.05353},
  year    = {2026}
}

@article{kitaoka2025inverse,
  author  = {Kitaoka, Akira},
  title   = {Inverse Mixed-Integer Programming: Learning Constraints then Objective Functions},
  journal = {arXiv preprint arXiv:2510.04455},
  year    = {2025}
}

@article{kumar2021learning,
  author  = {Kumar, Mohit and Kolb, Samuel and De Raedt, Luc and Teso, Stefano},
  title   = {Learning Mixed-Integer Linear Programs from Contextual Examples},
  journal = {arXiv preprint arXiv:2107.07136},
  year    = {2021}
}

@inproceedings{ren2025feasible,
  author    = {Ren, Ke and Mohajerin Esfahani, Peyman and Georghiou, Angelos},
  title     = {Inverse Optimization via Learning Feasible Regions},
  booktitle = {Proceedings of the 42nd International Conference on Machine Learning},
  series    = {Proceedings of Machine Learning Research},
  volume    = {267},
  pages     = {51471--51488},
  year      = {2025}
}

@inproceedings{tan2020learning,
  author    = {Tan, Yingcong and Terekhov, Daria and Delong, Andrew},
  title     = {Learning Linear Programs from Optimal Decisions},
  booktitle = {Advances in Neural Information Processing Systems},
  volume    = {33},
  year      = {2020}
}

@article{moghaddass2021multiple,
  author  = {Moghaddass, Mahsa and Terekhov, Daria},
  title   = {Inverse Integer Optimization with Multiple Observations},
  journal = {Optimization Letters},
  volume  = {15},
  number  = {4},
  pages   = {1061--1079},
  year    = {2021},
  doi     = {10.1007/s11590-021-01721-4}
}

@article{bulut2021complexity,
  author  = {Bulut, Aykut and Ralphs, Ted K.},
  title   = {On the Complexity of Inverse Mixed Integer Linear Optimization},
  journal = {SIAM Journal on Optimization},
  volume  = {31},
  number  = {4},
  pages   = {3014--3043},
  year    = {2021},
  doi     = {10.1137/20M1377369}
}

@article{aswani2018noisy,
  author  = {Aswani, Anil and Shen, Zuo-Jun Max and Siddiq, Auyon},
  title   = {Inverse Optimization with Noisy Data},
  journal = {Operations Research},
  volume  = {66},
  number  = {3},
  pages   = {870--892},
  year    = {2018},
  doi     = {10.1287/opre.2017.1705}
}

@article{mohajerin2018imperfect,
  author  = {Mohajerin Esfahani, Peyman and Shafieezadeh-Abadeh, Soroosh and Hanasusanto, Grani A. and Kuhn, Daniel},
  title   = {Data-Driven Inverse Optimization with Imperfect Information},
  journal = {Mathematical Programming},
  volume  = {167},
  number  = {1},
  pages   = {191--234},
  year    = {2018},
  doi     = {10.1007/s10107-017-1216-6}
}

@article{maragno2025constraint,
  author  = {Maragno, Donato and Wiberg, Holly and Bertsimas, Dimitris and Birbil, S. Ilker and den Hertog, Dick and Fajemisin, Adejuyigbe O.},
  title   = {Mixed-Integer Optimization with Constraint Learning},
  journal = {Operations Research},
  volume  = {73},
  number  = {2},
  pages   = {1011--1028},
  year    = {2025},
  doi     = {10.1287/opre.2023.2523}
}

@article{fajemisin2021survey,
  author  = {Fajemisin, Adejuyigbe and Maragno, Donato and den Hertog, Dick},
  title   = {Optimization with Constraint Learning: A Framework and Survey},
  journal = {arXiv preprint arXiv:2110.02121},
  year    = {2021}
}

@article{baxi2026online,
title = {Online Rack Placement in Large-Scale Data Centers: Online Sampling Optimization and Deployment},
author = {Baxi, Saumil and Cummings, Kayla and Jacquillat, Alexandre and Lo, Sean and McDonald, Rob and Mellou, Konstantina and Menache, Ishai and Molinaro, Marco},
journal = {Operations Research},
year = {2026},
doi = {10.1287/opre.2025.1992}
}

@inproceedings{chainofexperts2024,
  author    = {Xiao, Ziyang and Zhang, Dongxiang and Wu, Yangjun and Xu, Lilin and Wang, Yuan and Han, Xiongwei and Fu, Xiaojin and Zhong, Tao and Zeng, Jia and Song, Mingli and Chen, Gang},
  title     = {Chain-of-Experts: When {LLM}s Meet Complex Operations Research Problems},
  booktitle = {The Twelfth International Conference on Learning Representations},
  year      = {2024}
}

@inproceedings{llmopt2025,
  author    = {Jiang, Caigao and Shu, Xiang and Qian, Hong and Lu, Xingyu and Zhou, Jun and Zhou, Aimin and Yu, Yang},
  title     = {{LLMOPT}: Learning to Define and Solve General Optimization Problems from Scratch},
  booktitle = {The Thirteenth International Conference on Learning Representations},
  year      = {2025}
}

@article{xiao2025survey,
  author  = {Xiao, Ziyang and Xie, Jingrong and Xu, Lilin and Guan, Shisi and Zhu, Jingyan and Han, Xiongwei and Fu, Xiaojin and Yu, WingYin and Wu, Han and Shi, Wei and Kang, Qingcan and Duan, Jiahui and Zhong, Tao and Yuan, Mingxuan and Zeng, Jia and Wang, Yuan and Chen, Gang and Zhang, Dongxiang},
  title   = {A Survey of Optimization Modeling Meets {LLM}s: Progress and Future Directions},
  journal = {arXiv preprint arXiv:2508.10047},
  year    = {2025}
}

@article{besbes2025contextual,
  title={Contextual inverse optimization: Offline and online learning},
  author={Besbes, Omar and Fonseca, Yuri and Lobel, Ilan},
  journal={Operations Research},
  volume={73},
  number={1},
  pages={424--443},
  year={2025},
  publisher={INFORMS}
}
\bibliographystyle{plainnat}

\appendix
\setcounter{table}{0}
\setcounter{figure}{0}
\setcounter{algorithm}{0}
\renewcommand{\thetable}{\Alph{section}.\arabic{table}}
\renewcommand{\thefigure}{\Alph{section}.\arabic{figure}}
\renewcommand{\thealgorithm}{\Alph{section}.\arabic{algorithm}}
\makeatletter
\@addtoreset{table}{section}
\@addtoreset{figure}{section}
\@addtoreset{algorithm}{section}
\makeatother

\clearpage

\section{Proofs}

\subsection{Limitation of embedded ML models for formulation correction} \label{app:embed}

As mentioned in Section~\ref{sec:related}, one approach for correcting formulations is by using an ML model to represent the user action of mapping from ``incorrect'' to ``target'' solution, and embedding this model within the optimization itself~\citep{hewitt2020data,bayani2024learning,abdellaoui2026implicit}. That is, given an initial formulation \[ \M_0(d) = \min\left\{ c_0(d)^\top x : x\in\X_0(d) \right\}, \] consider the modified problem \[ \M_{0,f}(d) = \min\left\{ c_0(d)^\top x : x=f(x),\ x\in\X_0(d) \right\}, \] where $f:\mathbb{R}^n\rightarrow\mathbb{R}^n$ is a machine-learning model representing the user's correction of a solution \(x\) into a solution \(f(x)\). The function \(f\) is learned so that, ideally, the optimal solutions of \(\M_{0,f}(d)\) match the optimal solutions of the target formulation \(\M^\star(d)\) for every instance: \[ \argmin \M_{0,f}(d) = \Sopt(d;\M^\star) \qquad \text{for all } d. \] Some approaches relax the fixed-point equality \(x=f(x)\) through a Lagrangian penalty, i.e., the objective becomes $c_0(d)^\top x - \lambda \|x - f(x)\|_{\infty}$, but we stick to the original fixed-point equality to avoid the extra parameter $\lambda$.

The following result shows that this approach cannot represent all formulation corrections. In particular, the limitation already arises when the initial formulation is missing a nonlinear term in the objective. Although such a term can be represented in a MILP using auxiliary variables and constraints, it cannot, in general, be recovered by restricting the feasible set of the original variables. In fact, the impossibility persists even if new constraints are added. We note that through the use of templates, our proposed system is able to handle such corrections.

 \begin{theorem} \label{thm:embed} There exist an initial formulation \(\M_0\) and a target formulation \(\M^\star\) such that no function \(f\) and no set \(K\) of additional constraints make the optimal solutions of \[ \M_{0,f,K}(d) \coloneqq \min\left\{ c_0(d)^\top x : x=f(x),\ x\in\X_0(d),\ x\in K \right\} \] equal to the optimal solutions of \(\M^\star(d)\) for every instance \(d\). 
 \end{theorem}

\begin{proof}
Consider the family of instances parametrized by the scalar \(d\in\mathbb{R}\). The target formulation is \[ \M^\star(d) = \min\left\{ dx+|x| : x\in\{-1,0,1\} \right\}, \] whereas the initial formulation, which omits the nonlinear objective term, is \[ \M_0(d) = \min\left\{ dx : x\in\{-1,0,1\} \right\}. \] The target formulation can equivalently be written as the MILP \[ \begin{aligned} \M^\star(d):\qquad \min_{x,y_1,y_2}\quad & dx+y_1+y_2\\ \text{s.t.}\quad & x=y_1-y_2,\\ & x\in\{-1,0,1\},\\ & y_1,y_2\in\{0,1\}. \end{aligned} \] Thus, \(\M^\star\) can be obtained from \(\M_0\) by introducing auxiliary variables and constraints together with the missing objective term.

For a function $f : \R \rightarrow \R$ and set $K \subseteq \R$, in this case we have \[ \M_{0,f,K}(d) = \min\left\{ dx : x=f(x),\ x\in\{-1,0,1\},\ x\in K \right\}. \] We claim that for every $f$ and $K$, there exists some $d$ such that \[ \argmin \M_{0,f,K}(d) \neq \Pi_0\!\left(\Sopt(d;\M^\star)\right). \] 
To see that, let \[ S \coloneqq \left\{ x: x=f(x),\ x\in\{-1,0,1\},\ x\in K \right\} \] be the feasible set of \(\M_{0,f,K}(d)\). Importantly, $S$ is independent of $d$, and hence $\M_{0,f,K}(d) = \min_{x\in S} dx.$

    Consider $d \in \{-2,0,2\}$. For these instances, the target optimal solutions are \[ \Sopt(-2;\M^\star)=\{1\}, \qquad \Sopt(0;\M^\star)=\{0\}, \qquad \Sopt(2;\M^\star)=\{-1\}. \] 

Suppose by contradiction that the optimal-solution sets of $\M_{0,f,K}(d)$ and $\M^\star(d)$ agree for all three values of $d$. Agreement at $d=-2$ requires 1 to be feasible for \(\M_{0,f,K}(-2)\), and therefore \(1\in S\). Similarly, agreement at \(d=2\) requires \(-1\in S\).  When \(d=0\), however, the objective of \(\M_{0,f,K}(0)\) is identically zero over \(S\), so every feasible point is optimal: \[ \argmin \M_{0,f,K}(0)=S. \] Since \(-1,1\in S\), it follows that \[ \{-1,1\} \subseteq \argmin \M_{0,f,K}(0). \] This contradicts $\Sopt(0;\M^\star)=\{0\}$. Therefore no such $f$ and $K$ exist.
\end{proof}

\section{Extended Numerical Results} \label{app:extended_results}

\subsection{Additional Results}

\begin{table}[h]
\setlength{\tabcolsep}{4pt} 
\caption{Test optimality rate per misspecification type, with DeepSeek-V4-Flash as LLM backbone.}
\label{tab:methods-misspec}
\centering
\footnotesize
\begin{tabular}{lccccccccc}
\toprule
 & \multicolumn{2}{c}{Objective} & \multicolumn{3}{c}{Constraints} & \multicolumn{2}{c}{Variables} & Constants & \\
\cmidrule(lr){2-3} \cmidrule(lr){4-6} \cmidrule(lr){7-8} \cmidrule(lr){9-9}
 & Miscal. & Missing & Local & Global & Extra & Missing & Domain & Missing & All \\
Instances & \textit{7} & \textit{6} & \textit{10} & \textit{15} & \textit{5} & \textit{10} & \textit{2} & \textit{16} & \textit{38} \\
\midrule
Initial model & 0.24 & 0.22 & 0.20 & 0.11 & 0.17 & 0.18 & 0.01 & 0.13 & 0.17 \\
\midrule
Calib. only & 0.49 & 0.26 & 0.20 & 0.11 & 0.28 & 0.21 & 0.01 & 0.13 & 0.23 \\
Cuts & 0.52 & 0.20 & 0.41 & 0.07 & 0.40 & 0.12 & 0.00 & 0.06 & 0.27 \\
Cuts + contextualize & 0.52 & 0.22 & 0.43 & 0.11 & 0.41 & 0.13 & 0.00 & 0.08 & 0.30 \\
Template & 0.59 & 0.37 & 0.29 & 0.24 & 0.50 & 0.31 & \textbf{1.00} & 0.17 & 0.39 \\
Template + complete & 0.60 & \textbf{0.46} & 0.55 & \textbf{0.49} & \textbf{0.77} & 0.43 & \textbf{1.00} & \textbf{0.51} & 0.58 \\
All & \textbf{0.61} & 0.43 & \textbf{0.72} & 0.47 & 0.68 & \textbf{0.44} & \textbf{1.00} & \textbf{0.51} & \textbf{0.61} \\
\bottomrule
\end{tabular}
\end{table}

\begin{table}[h]
\caption{Recovery rate of the final formulation against the number of training instances per scenario.}
\label{tab:sizes}
\centering
\footnotesize
\begin{tabular}{lcccc}
\toprule
Training instances & 32 & 64 & 128 & 256 \\
\midrule
Calibration & 4.7\% & 5.8\% & 5.3\% & 6.3\% \\
Cuts + contextualize & 7.9\% & 13.7\% & 17.9\% & 20.5\% \\
Template + complete & \textbf{36.8\%} & 43.2\% & 41.1\% & 40.5\% \\
All & 36.3\% & \textbf{45.8\%} & \textbf{45.3\%} & \textbf{43.2\%} \\
\bottomrule
\end{tabular}
\end{table}

Table~\ref{tab:subopt} compares our framework trained on optimal targets with training on targets that are up to 5\% suboptimal. Let $Z^\star(d)$ be the optimal value of $\mathcal{M}^\star(d)$ and, for a solution $y$, let $Z(y;d)=c^\star(d)^\top y$. For each instance whose target differs from the initial solution, we draw $x \sim U(0,\Delta)$ with $\Delta=0.05$ and replace the target by a maximizer of $Z(y;d)$ over $\{y\in\mathcal{X}^\star(d): Z(y;d)\le Z^\star(d)+x|Z^\star(d)|\}$, i.e., the worst target-feasible solution within relative gap $x$. When $x^0$ admits a feasible lifting under $\mathcal{M}^\star(d)$ with best objective $Z_{\text{init}}(d)$, we add $Z(y;d)\le Z_{\text{init}}(d)-\delta|Z^\star(d)|$ with $\delta=0.005$, so the new target still beats the initial solution by a relative margin of $\delta$ (instances where this set is empty are discarded).

\begin{table}[h]
\caption{Performance of \textsc{TACIT} with DeepSeek-V4-Flash backbone on datasets with suboptimal target solutions.}
\label{tab:subopt}
\centering
\footnotesize
\begin{tabular}{lccccccc}
\toprule
 & \multicolumn{3}{c}{Training} & \multicolumn{4}{c}{Test} \\
\cmidrule(lr){2-4} \cmidrule(lr){5-8}
Dataset & $\L$ $\downarrow$ & feas. $\uparrow$ & opt. $\uparrow$ & feas. $\uparrow$ & opt. $\uparrow$ & rec. \% $\uparrow$ & recoverable $\uparrow$ \\
\midrule
Optimal targets & 0.19 & 0.79 & 0.64 & 0.78 & 0.61 & 45.3\% & 63.2\% \\
Suboptimal targets & 0.19 & 0.81 & 0.58 & 0.79 & 0.53 & 40.5\% & 52.6\% \\
\bottomrule
\end{tabular}
\end{table}

The Adapted Optimization Problem \citep{hewitt2020data, bayani2024learning} learns a mapping from the initial model's optimum to the observed plan and penalizes deviations from that prediction, so it can only imitate the observed plans rather than recover the rule that produced them. In Table~\ref{tab:AOP}, raising $c$ makes its solutions satisfy the hidden constraint more often (feasibility from 0.51 to 0.62 and 0.52 to 0.64) but only by pulling them toward typical plans, so they stop being optimal (optimality from 0.13 to 0.01 and 0.15 to 0.04) and no run ever recovers the ground truth.

\begin{table}[h]
\caption{AOP methods across different values of the penalty parameter $c$.}
\centering
\label{tab:AOP}
\footnotesize
\setlength{\tabcolsep}{4pt}
\begin{tabular}{llcccccccc}
\toprule
 & & \multicolumn{4}{c}{Training} & \multicolumn{4}{c}{Test} \\
\cmidrule(lr){3-6} \cmidrule(lr){7-10}
 & Model & $\L$ $\downarrow$ & sep. $\uparrow$ & feas. $\uparrow$ & opt. $\uparrow$ & feas. $\uparrow$ & opt. $\uparrow$ & rec. \% $\uparrow$ & recoverable $\uparrow$ \\
\midrule
 & Initial & \textbf{0.85} & 0.0\% & 0.46 & \textbf{0.17} & 0.46 & \textbf{0.17} & - & - \\
 \addlinespace
 & AOP linear ($c{=}1$) & 0.90 & 0.0\% & 0.50 & 0.15 & 0.51 & 0.13 & 0.0\% & 0.0\% \\
 & AOP linear ($c{=}16$) & 0.94 & 0.0\% & 0.57 & 0.07 & 0.57 & 0.04 & 0.0\% & 0.0\% \\
 & AOP linear ($c{=}256$) & 0.98 & 0.0\% & 0.61 & 0.02 & 0.62 & 0.01 & 0.0\% & 0.0\% \\
  \addlinespace
 & AOP tree ($c{=}1$) & 0.91 & 0.0\% & 0.52 & \textbf{0.17} & 0.52 & 0.15 & 0.0\% & 0.0\% \\
 & AOP tree ($c{=}16$) & 0.93 & 0.0\% & 0.58 & 0.10 & 0.57 & 0.07 & 0.0\% & 0.0\% \\
 & AOP tree ($c{=}256$) & 0.95 & 0.0\% & \textbf{0.65} & 0.07 & \textbf{0.64} & 0.04 & 0.0\% & 0.0\% \\
\bottomrule
\end{tabular}
\end{table}

\subsection{Example Repaired Formulation}

\subsubsection{Bottom-Up}

Table~\ref{tab:cut-examples} shows two cases where refining a cut recovers the missing constraint exactly. In GPU Scenario 3, the raw cut has near-degenerate coefficients. Because the variables $z_i$ are binary, it still encodes the implication $z_0 = 1 \Rightarrow z_i = 1$ for all $i>0$. LLM rewrites it as the family in the more interpretable form $z_0 \leq z_i$ for all $i > 0$. In contrast, the template-based search never recovers a correct formulation on this scenario. In WA Scenario 3, the raw cut is already interpretable but applies to a single day and is not sufficient for out-of-sample test instances, so that raw cuts reach a test optimality of $0$. The LLM refinement keeps employees $3$ and $5$ as fixed anchors, extends the cut over the day index $d$, and recovers the ground-truth constraint exactly.

\begin{table}[h]
\caption{Example of cut processing on two scenarios.}
\label{tab:cut-examples}
\centering
\small
\renewcommand{\arraystretch}{1.5}
\setlength{\tabcolsep}{6pt}
\begin{tabular}{@{}l l@{}}
\toprule
\multicolumn{2}{@{}l}{\textbf{GPU Scenario 3:} \textit{demand 0 can only be satisfied whenever all other demands are satisfied.}} \\
Missing constraint: & $z_0 \leq z_i, \ \forall i \in I, \ i > 0$ \\
Raw cut from \eqref{eq:cut}:            & $0.9999\, z_0 - 0.0001\, z_1 - 0.0001\, z_2 - 0.9996\, z_3 - 0.0001\, z_4 \leq 0$ \\
LLM refined cut:       & $z_0 \leq z_i, \ \forall i \in I, \ i > 0$ \\
\midrule
\multicolumn{2}{@{}l}{\textbf{WA Scenario 3:} \textit{employees 3 and 5 can never work on the same day.}} \\
Missing constraint: & $x_{3,d} + x_{5,d} \leq 1, \ \forall d \in D$ \\
Raw cut from \eqref{eq:cut}:            & $0.5\, x_{3,7} + 0.5\, x_{5,7} - 0.5 \leq 0$ \\
LLM refined cut:    & $x_{3,d} + x_{5,d} \leq 1, \ \forall d \in D$ \\
\bottomrule
\end{tabular}
\end{table}

\subsubsection{Top-Down}

Table~\ref{tab:template-examples} shows three templates proposed by the LLM that recover the ground truth, and how the fitting step turns each into a concrete correction. In VRP Scenario 2, the template adds a variable $z$ that bounds every route cost. It has no free constants so no fitting in necessary. In TP Scenario 3, the template adds both a new integer variable $y_{i,j}$ and a new constant, the batch size $L$. The constraint calibration \eqref{eq:constraint-calibration} recovers the true value $L = 2$. In CFL Scenario 7, the LLM proposes a general per-pair cap $x_{i,j} \leq u_{i,j}\, y_i$ without knowing which facility is restricted. Here, the constraint calibration \eqref{eq:constraint-calibration} fit the constants $u_{i,j}$ to capture the local rule: every cap is set to $1$, which leaves the original constraint unchanged, except for facility $4$, whose cap is set to $0.5$.

\begin{table}[h]
\caption{Example of template processing on three scenarios. Templates proposed by the LLM, before (raw) and after fitting the new constants with \eqref{eq:constraint-calibration}. Fitted values are shown in blue.}
\label{tab:template-examples}
\centering
\small
\renewcommand{\arraystretch}{1.5}
\setlength{\tabcolsep}{6pt}
\begin{tabular}{@{}l c c@{}}
\toprule
 & \textbf{Template (raw)} & \textbf{Template (fitted)} \\
\midrule
\multicolumn{3}{@{}l}{\textbf{CVRP Scenario 2:} \textit{the ground truth minimizes the longest route instead of the total distance.}} \\
Variables:   & $z \in \mathbb{R}_{\geq 0}$ & $z \in \mathbb{R}_{\geq 0}$ \\
Constants:   & -- & -- \\
Constraints: & $\displaystyle\sum_{i \in N} \sum_{\substack{j \in N \\ j \neq i}} c_{i,j}\, x_{i,j,k} \leq z, \ \forall k \in K$
            & $\displaystyle\sum_{i \in N} \sum_{\substack{j \in N \\ j \neq i}} c_{i,j}\, x_{i,j,k} \leq z, \ \forall k \in K$ \\
\midrule
\multicolumn{3}{@{}l}{\textbf{TP Scenario 3:} \textit{goods are shipped in batches of a fixed size.}} \\
Variables:   & $y_{i,j} \in \mathbb{Z}, \ \forall i \in I, j \in J$ & $y_{i,j} \in \mathbb{Z}, \ \forall i \in I, j \in J$ \\
Constants:   & $L$ & $L = {\color{blue}2}$ \\
Constraints: & $x_{i,j} = L\, y_{i,j}, \ \forall i \in I, j \in J$
            & $x_{i,j} = {\color{blue}2}\, y_{i,j}, \ \forall i \in I, j \in J$ \\
\midrule
\multicolumn{3}{@{}l}{\textbf{CFL Scenario 7:} \textit{facility 4 can serve at most half of the demand of any customer.}} \\
Variables:   & -- & -- \\
Constants:   & $u_{i,j}, \ \forall i \in I, j \in J$
            & $u_{4,j} = {\color{blue}0.5}, \ u_{i,j} = {\color{blue}1} \ (i \neq 4)$ \\
Constraints: & $x_{i,j} \leq u_{i,j}\, y_i, \ \forall i \in I, j \in J$
            & $x_{4,j} \leq {\color{blue}0.5}\, y_4, \ \forall j \in J$ \\
            & & $x_{i,j} \leq y_i, \ \forall i \in I \setminus \{4\}, j \in J$ \\
\bottomrule
\end{tabular}
\end{table}

\subsubsection{Final formulation samples} \label{sec:final_form_samples}

In these examples, $\mathcal{M}'$ is the final formulation returned by TACIT.

\paragraph{TP Scenario 2.}
At most $K=2$ origins may ship to a given destination. Here \textsc{TACIT} recovers the target formulation exactly, up to the big-$M$ constant ($24$ instead of $M=1000$):

\vspace{-0.6cm}
\begin{center}
\scalebox{0.85}{\begin{minipage}[t]{0.48\textwidth} \centering \[ \begin{aligned} \mathcal M^\star:\quad \min_{x,z}\quad & \sum_{i\in I}\sum_{j\in J} c_{ij}x_{ij}\\ \text{s.t.}\quad & x\in\mathcal X_{\mathrm{TP}},\\ & \blue{x_{ij} \le M z_{ij},} && \blue{\forall i\in I,\ j\in J,}\\ & \blue{\sum_{i\in I} z_{ij} \le K,} && \blue{\forall j\in J,}\\ & \blue{z_{ij}\in\{0,1\},} && \blue{\forall i\in I,\ j\in J.} \end{aligned} \] \end{minipage}
\rule[-3.8cm]{0.4pt}{3.8cm}
\hspace{0.1cm}
\begin{minipage}[t]{0.48\textwidth} \centering \[ \begin{aligned} \mathcal M':\quad \min_{x,y}\quad & \sum_{i\in I}\sum_{j\in J} c_{ij}x_{ij}\\ \text{s.t.}\quad & x\in\mathcal X_{\mathrm{TP}},\\ & \green{x_{ij} \le 24\, y_{ij},} && \green{\forall i\in I,\ j\in J,}\\ & \green{\sum_{i\in I} y_{ij} \le 2,} && \green{\forall j\in J,}\\ & \green{y_{ij}\in\{0,1\},} && \green{\forall i\in I,\ j\in J.} \end{aligned} \] \end{minipage} }
\end{center}

\paragraph{GPU Scenario 4.}
The target values the total demand fulfilled rather than the number of requests fulfilled. Since $\mathcal X_{\mathrm{GPU}}$ contains $\sum_{j\in J_i} x_{ij} = s_i z_i$, the incumbent objective is equivalent to $-4.32\sum_{i\in I} s_i z_i - 0.10\sum_{i\in I} z_i$, i.e. the target objective scaled by $4.32$ plus a small tie-breaker on the number of requests.
\begin{center}
\scalebox{0.85}{\begin{minipage}[t]{0.40\textwidth} \centering \[ \begin{aligned} \mathcal M^\star:\quad \min_{x,y,z}\quad & -\sum_{i\in I} \blue{s_i}\, z_i\\ \text{s.t.}\quad & (x,y,z)\in\mathcal X_{\mathrm{GPU}}. \end{aligned} \] \end{minipage}
\rule[-2.2cm]{0.4pt}{1.8cm}
\hspace{0.1cm}
\begin{minipage}[t]{0.60\textwidth} \centering \[ \begin{aligned} \mathcal M':\quad \min_{x,y,z}\quad & \green{-1.90\sum_{i\in I} s_i z_i - 0.10\sum_{i\in I} z_i} \green{ - 2.42\sum_{i\in I}\sum_{j\in J_i} x_{ij}}\\ \text{s.t.}\quad & (x,y,z)\in\mathcal X_{\mathrm{GPU}}. \end{aligned} \] \end{minipage} }
\end{center}

\paragraph{MCF Scenario 2.}
Minimize the number of used arcs. Here TACIT also recovers precisely an equivalent formulation (we note that $f \in \mathcal X_{\mathrm{MCF}}$ guarantees that $f \ge 0$, and so the last inequality in $\mathcal{M}'$ is valid, albeit redundant):
\begin{center}
\scalebox{0.85}{\begin{minipage}[t]{0.48\textwidth} \centering \[ \begin{aligned} \mathcal M^\star:\quad \min_{f,w}\quad & \blue{\sum_{i\in V}\sum_{j\in O_i} w_{ij}}\\ \text{s.t.}\quad & f\in\mathcal X_{\mathrm{MCF}},\\ & \blue{f_{ij} \le u_{ij} w_{ij},} && \blue{\forall i\in V,\ j\in O_i,}\\ & \blue{w_{ij}\in\{0,1\},} && \blue{\forall i\in V,\ j\in O_i.} \end{aligned} \] \end{minipage}
\hspace{0.1cm}
\rule[-3.3cm]{0.4pt}{3cm}
\hspace{0.1cm}
\begin{minipage}[t]{0.48\textwidth} \centering \[ \begin{aligned} \mathcal M':\quad \min_{f,y}\quad & \green{\sum_{i\in V}\sum_{j\in O_i} y_{ij}}\\ \text{s.t.}\quad & f\in\mathcal X_{\mathrm{MCF}},\\ & \green{f_{ij} \le u_{ij} y_{ij},} && \green{\forall i\in V,\ j\in O_i,}\\ & \green{0\cdot y_{ij} \le f_{ij},} && \green{\forall i\in V,\ j\in O_i,}\\ & \green{y_{ij}\in\{0,1\},} && \green{\forall i\in V,\ j\in O_i.} \end{aligned} \] \end{minipage} }
\end{center}

\section{Implementation Details} \label{app:implementation_details}

\subsection{Routine and Subroutines} \label{app:subroutines}

Algorithms~\ref{alg:full_framework} to \ref{alg:templates} summarize the framework of Section~\ref{sec:methodology}. We write $\M \oplus \cdot$ for the formulation obtained by adding constraints, variables, or an instantiated template to $\M$, and $\mathcal{C}$ for the cuts accumulated across iterations. A single incumbent $\M'$ is maintained: every candidate formulation is scored by \textsc{Evaluate}, which compares it to the incumbent before and after objective calibration and keeps the best. Cuts are searched on $\M' \oplus \mathcal{C}$ so that successive calls separate the solutions surviving the previous cuts, while templates are proposed from $\M_0$ to avoid compounding noise from previous modifications.

\begin{algorithm}[h]
\caption{\frameworkName Framework}
\label{alg:full_framework}
\begin{algorithmic}[1]
\REQUIRE initial formulation $\M_0$, data $\D$, iteration budget $T$
\STATE $\M' \gets \M_0$ with every constraint violated by some $x_i^\star$ removed \hfill\emph{(preprocessing)}
\STATE $\M' \gets$ \textsc{Evaluate}$(\M', \M')$
\STATE $\mathcal{C} \gets \emptyset$
\FOR{$t = 1, \dots, T$ \textbf{while} $\L(\M';\D) > 0$}
    \STATE $\M', \mathcal{C} \gets$ \textsc{SearchCuts}$(\M', \mathcal{C})$ \hfill\emph{(Sec.~\ref{sec:bottom-up})}
    \STATE $\M' \gets$ \textsc{SearchTemplates}$(\M', \M_0)$ \hfill\emph{(Sec.~\ref{sec:top-down})}
\ENDFOR
\RETURN $\M'$
\end{algorithmic}
\end{algorithm}

\begin{algorithm}[h]
\caption{\textsc{Evaluate}$(\M', \M)$: score a candidate $\M$ and update the incumbent $\M'$}
\label{alg:evaluate}
\begin{algorithmic}[1]
\IF{some $x_i^\star$ has no lifting under $\M$}
    \RETURN $\M'$
\ENDIF
\IF{$\L(\M;\D) < \L(\M';\D)$}
    \STATE $\M' \gets \M$
\ENDIF
\STATE $\theta \gets$ solution of \eqref{eq:objective-calibration} on $\M$ \hfill\emph{(Sec.~\ref{sec:objective-calibration})}
\STATE $\M_\theta \gets (c_\theta, \X)$
\IF{$\L(\M_\theta;\D) < \L(\M';\D)$}
    \STATE $\M' \gets \M_\theta$
\ENDIF
\RETURN $\M'$
\end{algorithmic}
\end{algorithm}

\begin{algorithm}[h]
\caption{\textsc{SearchCuts}$(\M', \mathcal{C})$: cut search and LLM refinement}
\label{alg:cuts}
\begin{algorithmic}[1]
\STATE $\M \gets \M' \oplus \mathcal{C}$
\FOR{$\rho$ in $\{2^7, \dots , 2^{-7}\}$}
    \STATE $\{(a^{(0)}, b^{(0)})\} \gets$ solution of the cut MIP on $\M$ with penalty $\rho$ \hfill\emph{(App.~\ref{app:cuts})}
    \STATE $\{(a^{(k)}, b^{(k)})\}_{k \ge 1} \gets$ LLM-polished variants of $(a^{(0)}, b^{(0)})$ given $\M$
    \FOR{$k = 0, 1, \dots$}
        \STATE $\M' \gets$ \textsc{Evaluate}$\big(\M',  \M \oplus \{a^{(k)\top} x \le b^{(k)}\}\big)$
        \STATE $\mathcal{C} \gets \mathcal{C} \cup \{a^{(k)\top} x \le b^{(k)}\}$
    \ENDFOR
\ENDFOR
\STATE $\M' \gets$ \textsc{Evaluate}$(\M', \M' \oplus \mathcal{C})$
\FOR{each constraint family $\mathcal{G}$ returned by the LLM as an extrapolation of $\mathcal{C}$}
    \STATE $\M' \gets$ \textsc{Evaluate}$(\M', \M' \oplus \mathcal{G})$
\ENDFOR
\RETURN $\M'$
\end{algorithmic}
\end{algorithm}

\begin{algorithm}[h]
\caption{\textsc{SearchTemplates}$(\M', \M)$: template generation and calibration}
\label{alg:templates}
\begin{algorithmic}[1]
\STATE $\pi \gets$ sample $n_\text{pairs}$ mismatched pairs $(\hat x_i, x_i^\star)$ with $\hat x_i \in \Pi_0(\Sopt(d_i;\M))$, $\hat x_i \ne x_i^\star$
\STATE $\{\tau_1, \dots, \tau_m\} \gets$ templates returned by the LLM given $\M$ and $\pi$
\FOR{$j = 1, \dots, m$}
    \IF{\eqref{eq:constraint-calibration} is infeasible}
        \STATE continue
    \ENDIF
    \STATE $w \gets$ solution of \eqref{eq:constraint-calibration} on $\M \oplus \tau_j$
    \STATE $\M' \gets$ \textsc{Evaluate}$(\M', \M \oplus \tau_j(w))$
\ENDFOR
\RETURN $\M'$
\end{algorithmic}
\end{algorithm}

\subsection{Hyperparameters} \label{app:hyperparameters}

Each run uses $N = 128$ training instances unless stated otherwise and is evaluated on 256 test instances. We use rejection sampling to ensure that at least half of target solutions differ from the initial ones. The search runs for at most 5 iterations and stops early once the training loss is optimal. Cuts are generated by the hyperplane-separation MIP with an $\ell_0$ penalty starting at 128 and halved until non-empty cuts are found. Template prompts show up to 16 model/target solution pairs, the constraints removed during preprocessing, and up to 5 templates that were returned by the LLM on previous iterations to avoid repetitions. All MIPs are solved with Gurobi on a single thread with a 30-second limit per instance. LLMs are queried with high reasoning effort and default temperature.

\ifdefined\isarxiv
\else

\subsection{Prompt Templates} \label{app:prompts}

We report the prompts verbatim. Fields between single angle brackets (e.g.\texttt{<skeleton\_str>}) are filled at run time; tags between double angle brackets (e.g.\ \texttt{<<START\_MODIFICATION>>}) are the delimiters used to parse the response. The current formulation (\texttt{<skeleton\_str>}) is always given as gurobipy-style pseudo-code.

\paragraph{Cut refinement.} Each raw cut returned by the separation MIP has fitted numerical coefficients. The LLM rewrites it into simpler, interpretable variants (rounded coefficients, dropped terms, extension to a family), which are then evaluated like any other candidate cut.
\begin{promptbox}{Polish cut}
\VerbatimInput{assets/prompts_txt/polish_cut.txt}
\end{promptbox}

\paragraph{Cut generalization.} Given up to 16 cuts found so far (sparse cuts are prioritized), the LLM abstracts their common pattern into constraint families over the existing sets, using only existing symbols.
\begin{promptbox}{Generalize cuts}
\VerbatimInput{assets/prompts_txt/generalize_cuts.txt}
\end{promptbox}

\paragraph{Template proposal.} Given the initial formulation and a set of instances where its solutions differ from the targets (only the differing variables are shown), the LLM proposes modifications that may introduce new symbolic constants, whose values are then set by calibration. The prompt also lists the constraints removed by preprocessing and the most recent modifications, to avoid repeating them.
\begin{promptbox}{Get templates}
\VerbatimInput{assets/prompts_txt/get_templates.txt}
\end{promptbox}

\paragraph{Parsing and translation.} Three additional prompts, omitted here for brevity, convert the free-text modifications and constraints returned above into executable code: one splits a modification into its new constants, variables, type changes, and constraints, the two other translate each constraint and variable into our structured representation using few-shot examples. They perform a syntactic translation only and will be available once the code is released.

\paragraph{LLM-only baseline.} The LLM receives the initial formulation as compilable gurobipy code, the solution differences on all instance pairs, and the data format, and returns a complete corrected gurobipy model in a single query.
\begin{promptbox}{Get formulation (LLM-only baseline)}
\VerbatimInput{assets/prompts_txt/get_formulation.txt}
\end{promptbox}

\fi

\subsection{Compute and runtime} \label{app:runtime}

Runs are parallelized on an AMD EPYC 9734 processor, one run per core with Gurobi restricted to a single thread. LLMs are queried through hosted APIs with high reasoning effort. Each method is run 5 times per scenario, each run on a dataset generated with a different random seed, for at most 5 iterations and under a time budget of 6 hours per run. For datasets with $N{=}128$ instances, the median runtime per run ranges from 2 min for calibration alone to 0.8 hours for template, 1.3 hour for cuts and refine, and 2.1 hours for the full framework. The runtime of the LLM-only baseline is 0.2 hour for the LLM-only baseline.

We use two models in our framework: the large model, reported in the main text (DeepSeek-V4-Flash, Kimi-K2.6 or gpt-5.6-sol), generates templates and refine cuts, while the small model (always DeepSeek-V4-Flash) handles simpler tasks: split templates into variables, constants, and constraints, and translating natural-language modifications into optimization code. Averaged across the three large models, a run of our framework consumes 154k tokens from the large model and 346k tokens from the small model. The LLM-only baseline uses on average 210k tokens per run, all from the large model.

\section{Benchmark Applications and Scenarios}\label{app:problems_and_scenarios}

\ifdefined\isarxiv

The benchmark consists of 38 scenarios built on nine base problems. Each
scenario is a pair of formulations: the initial model used as the starting point, and the target model that generated the observed solutions.
Table~\ref{tab:scenario-taxonomy} classifies what separates the two. Below,
every problem is introduced with its notation, its base formulation and how
instances are drawn, and its scenarios are then listed. Unless we say
otherwise, the initial model is the base formulation and the target model is
the base formulation with the stated change. Every model is written as a
minimization, so maximization problems appear with a negated objective.
Scenarios that single out specific indices (e.g. facility 4, or employees 3 and 5) are run on instances of fixed size so that these indices always exist.
Numerical constants introduced by a target model ($R$, $L$, $K$, \dots) are
hidden from the repair method, which only sees the initial model and the
solutions.

\subsection{Capacitated facility location (CFL)}
\label{app:cfl}

Facilities are opened at a fixed cost and customers are assigned to open
facilities, possibly splitting their demand. Let $I$ be the facilities and $J$
the customers. Customer $j$ has demand $d_j$, facility $i$ has capacity $C_i$
and opening cost $f_i$, and serving the whole demand of $j$ from $i$ costs
$c_{ij}$. Variable $y_i\in\{0,1\}$ opens facility $i$ and $x_{ij}\in[0,1]$ is
the fraction of $j$'s demand served by $i$.
\begin{align}
    \min\quad & \sum_{i\in I} f_i y_i + \sum_{i\in I}\sum_{j\in J} c_{ij}x_{ij}
      \label{eq:cfl-obj}\\
    \text{s.t.}\quad
    & \sum_{i\in I}x_{ij}\ge 1, && \forall j\in J, \label{eq:cfl-demand}\\
    & \sum_{j\in J}d_jx_{ij}\le C_i y_i, && \forall i\in I, \label{eq:cfl-cap}\\
    & 0\le x_{ij}\le 1,\ y_i\in\{0,1\}, && \forall i\in I,\ j\in J.
\end{align}
Instances have 6 to 8 facilities and 8 to 14 customers, $d_j$ uniform in
$[1,20]$, $f_i$ in $[50,250]$ and $c_{ij}$ in $[1,100]$. Capacities are
$C_i=\lceil \alpha_i \bar d\rceil$ with $\bar d$ the total demand per facility
and $\alpha_i$ uniform in $[1.5,2.5]$, so that roughly half of the facilities
need to open.
\begin{itemize}
    \item \textbf{S1: service radius.}
    A facility can only serve customers within distance $R=50$. The target adds
    $z_{ij}\in\{0,1\}$ with $x_{ij}\le z_{ij}$ and $c_{ij}z_{ij}\le R$ for all
    $i,j$.
    \item \textbf{S2: overestimated capacity.}
    The target replaces~\eqref{eq:cfl-cap} by the tighter
    $\sum_j d_jx_{ij}\le 0.8\,C_i y_i$. The initial model uses the full $C_i$.
    \item \textbf{S3: underestimated capacity.}
    The target allows $\sum_j d_jx_{ij}\le 1.2\,C_i y_i$, so the capacity
    constraint of the initial model cuts off target solutions.
    \item \textbf{S4-S7:} Will be released with the open-source version of the benchmark.
\end{itemize}

\subsection{Other problems}
The remaining 8 problems and their scenarios will be released with the open-source version of the benchmark.

\else
  
\fi

\begin{table}[h]
\caption{Taxonomy of scenario misspecifications over the 38 benchmark scenarios spanning 9 problems: CFL (capacitated facility location), GPU (GPU allocation), MIS (maximum independent set), MCF (minimum-cost flow), MPP (multi-period production planning), SC (set cover), TP (transportation), CVRP (capacitated vehicle routing), WA (workforce assignment). A constraint is local when the restriction it adds involves fixed indices (e.g. $\sum_j \left(x_{3,j} + x_{5,j}\right) \leq 1$) and global when it is stated over generic indices (e.g. $\sum_{i,j} x_{i,j} \leq 1$). Extra marks a constraint of the initial model that cuts off target-feasible or target-optimal solutions. Constraints that only define an auxiliary variable are counted under missing variable rather than as a missing constraint, and big-$M$ bounds are not counted as missing constants.}
\label{tab:scenario-taxonomy}
    \centering
    \footnotesize
\begin{tabular}{cccccccccc}
\toprule
& & \multicolumn{2}{c}{Objective} & \multicolumn{3}{c}{Constraints} & \multicolumn{2}{c}{Variables} & Constants \\
\cmidrule(lr){3-4} \cmidrule(lr){5-7} \cmidrule(lr){8-9} \cmidrule(lr){10-10}
& Scenario & Miscalibrated & Missing & Local & Global & Extra & Missing & Domain & Missing \\
\midrule
\multirow{7}{*}{\rotatebox[origin=c]{90}{CFL}}
& S1 & $\square$ & $\square$ & $\square$ & $\blacksquare$ & $\square$ & $\blacksquare$ & $\square$ & $\blacksquare$ \\
& S2 & $\square$ & $\square$ & $\square$ & $\blacksquare$ & $\square$ & $\square$ & $\square$ & $\blacksquare$ \\
& S3 & $\square$ & $\square$ & $\square$ & $\blacksquare$ & $\blacksquare$ & $\square$ & $\square$ & $\blacksquare$ \\
& S4 & $\blacksquare$ & $\square$ & $\square$ & $\square$ & $\square$ & $\square$ & $\square$ & $\square$ \\
& S5 & $\blacksquare$ & $\blacksquare$ & $\square$ & $\square$ & $\square$ & $\blacksquare$ & $\square$ & $\square$ \\
& S6 & $\square$ & $\square$ & $\square$ & $\square$ & $\square$ & $\square$ & $\blacksquare$ & $\square$ \\
& S7 & $\square$ & $\square$ & $\blacksquare$ & $\square$ & $\square$ & $\square$ & $\square$ & $\blacksquare$ \\
\midrule
\multirow{5}{*}{\rotatebox[origin=c]{90}{GPU}}
& S1 & $\square$ & $\blacksquare$ & $\square$ & $\square$ & $\square$ & $\blacksquare$ & $\square$ & $\square$ \\
& S2 & $\square$ & $\square$ & $\square$ & $\blacksquare$ & $\square$ & $\square$ & $\square$ & $\square$ \\
& S3 & $\square$ & $\square$ & $\blacksquare$ & $\square$ & $\square$ & $\square$ & $\square$ & $\square$ \\
& S4 & $\blacksquare$ & $\square$ & $\square$ & $\square$ & $\square$ & $\square$ & $\square$ & $\square$ \\
& S5 & $\square$ & $\square$ & $\square$ & $\blacksquare$ & $\square$ & $\square$ & $\square$ & $\blacksquare$ \\
\midrule
\multirow{3}{*}{\rotatebox[origin=c]{90}{MIS}}
& S1 & $\square$ & $\square$ & $\square$ & $\square$ & $\blacksquare$ & $\square$ & $\square$ & $\square$ \\
& S2 & $\square$ & $\square$ & $\blacksquare$ & $\square$ & $\square$ & $\square$ & $\square$ & $\square$ \\
& S3 & $\square$ & $\square$ & $\square$ & $\blacksquare$ & $\square$ & $\blacksquare$ & $\square$ & $\blacksquare$ \\
\midrule
\multirow{2}{*}{\rotatebox[origin=c]{90}{MCF}}
& S1 & $\square$ & $\square$ & $\square$ & $\blacksquare$ & $\blacksquare$ & $\square$ & $\square$ & $\square$ \\
& S2 & $\blacksquare$ & $\blacksquare$ & $\square$ & $\square$ & $\square$ & $\blacksquare$ & $\square$ & $\square$ \\
\midrule
\multirow{2}{*}{\rotatebox[origin=c]{90}{MPP}}
& S1 & $\blacksquare$ & $\square$ & $\square$ & $\square$ & $\square$ & $\square$ & $\square$ & $\square$ \\
& S2 & $\square$ & $\square$ & $\square$ & $\blacksquare$ & $\square$ & $\square$ & $\square$ & $\blacksquare$ \\
\midrule
\multirow{5}{*}{\rotatebox[origin=c]{90}{SC}}
& S1 & $\square$ & $\square$ & $\blacksquare$ & $\square$ & $\square$ & $\square$ & $\square$ & $\blacksquare$ \\
& S2 & $\square$ & $\square$ & $\square$ & $\square$ & $\blacksquare$ & $\square$ & $\square$ & $\square$ \\
& S3 & $\square$ & $\square$ & $\blacksquare$ & $\square$ & $\blacksquare$ & $\square$ & $\square$ & $\square$ \\
& S4 & $\blacksquare$ & $\square$ & $\square$ & $\square$ & $\square$ & $\square$ & $\square$ & $\square$ \\
& S5 & $\square$ & $\square$ & $\square$ & $\blacksquare$ & $\square$ & $\square$ & $\square$ & $\blacksquare$ \\
\midrule
\multirow{6}{*}{\rotatebox[origin=c]{90}{TP}}
& S1 & $\square$ & $\square$ & $\square$ & $\blacksquare$ & $\square$ & $\square$ & $\square$ & $\blacksquare$ \\
& S2 & $\square$ & $\square$ & $\square$ & $\blacksquare$ & $\square$ & $\blacksquare$ & $\square$ & $\blacksquare$ \\
& S3 & $\square$ & $\square$ & $\square$ & $\blacksquare$ & $\square$ & $\blacksquare$ & $\square$ & $\blacksquare$ \\
& S4 & $\square$ & $\blacksquare$ & $\square$ & $\square$ & $\square$ & $\blacksquare$ & $\square$ & $\blacksquare$ \\
& S5 & $\square$ & $\square$ & $\square$ & $\square$ & $\square$ & $\square$ & $\blacksquare$ & $\square$ \\
& S6 & $\square$ & $\square$ & $\blacksquare$ & $\square$ & $\square$ & $\square$ & $\square$ & $\square$ \\
\midrule
\multirow{4}{*}{\rotatebox[origin=c]{90}{CVRP}}
& S1 & $\square$ & $\square$ & $\square$ & $\blacksquare$ & $\square$ & $\square$ & $\square$ & $\blacksquare$ \\
& S2 & $\blacksquare$ & $\blacksquare$ & $\square$ & $\square$ & $\square$ & $\blacksquare$ & $\square$ & $\square$ \\
& S3 & $\square$ & $\square$ & $\blacksquare$ & $\square$ & $\square$ & $\square$ & $\square$ & $\square$ \\
& S4 & $\square$ & $\square$ & $\blacksquare$ & $\square$ & $\square$ & $\square$ & $\square$ & $\square$ \\
\midrule
\multirow{4}{*}{\rotatebox[origin=c]{90}{WA}}
& S1 & $\square$ & $\square$ & $\square$ & $\blacksquare$ & $\square$ & $\square$ & $\square$ & $\blacksquare$ \\
& S2 & $\square$ & $\blacksquare$ & $\square$ & $\blacksquare$ & $\square$ & $\blacksquare$ & $\square$ & $\blacksquare$ \\
& S3 & $\square$ & $\square$ & $\blacksquare$ & $\square$ & $\square$ & $\square$ & $\square$ & $\square$ \\
& S4 & $\square$ & $\square$ & $\blacksquare$ & $\square$ & $\square$ & $\square$ & $\square$ & $\square$ \\
\bottomrule
\end{tabular}
\end{table}

\section{Additional Details}

\subsection{Cut Search} \label{app:cuts}

\begin{definition}[Variable family]
    Let $\mathcal{K}_i(\M)$ be the set of indices of the decision variables of $\M(d_i)$, whose size may vary across instances. Each variable is identified by a symbol and a tuple of indices ranging over sets that are either shared by all instances or instance-dependent. The \emph{family} $t(k)$ of variable $k$ is obtained by dropping its indices with instance-dependent sets. We denote by $\mathcal{T}(\M)$ the set of families. In particular, when every instance shares the same index sets, each variable forms its own family.
\end{definition}

The intuition behind families is that weights learned from data must be defined at the family level: a weight attached to a family applies to every instance, whereas a weight attached to an individual variable is meaningless on instances where that variable does not exist.

For instance, consider the VRP example (Example~\ref{ex:VRP}) where there is a fixed fleet of 5 vehicles $\mathcal{L}$ for all instances, but the customer set $C$ varies across instances. The routing variable $x_{ij \ell}$ indicates vehicle $\ell$ travels from customer $i$ to $j$. A variable index is a tuple $k = (i,j,l) \in N \times N \times \{1,\dots,5\}$, whose customer components $i,j$ are instance-dependent and whose vehicle component $\ell$ is shared. In this case, there is one family per vehicle, $\{x_{ijl} : (i,j) \in N \times N\}$ for $l \in \{1, \dots, 5\}$. In particular, it is not possible to learn customer-dependent weights that apply to all instances.

Let $\varphi_k(d)$ be the elements of $d$ whose instance-dependent indices all appear among the indices of variable $k$. Elements indexed only by index sets identical across instances belong to every $\varphi_k(d)$, and we denote this common part by $\varphi_0(d)$. Define $\phi_k(d) = \big(1, \varphi_k(d)\big)$. A cut then reads $a(d)^\top x \le b(d)$ with $a(d)_k = \alpha_{t(k)}^\top \phi_k(d)$ and $b(d) = \beta^\top \phi_0(d)$, where $\alpha_{t(k)}$ and $\beta$ are weights to be inferred. We solve the following MILP to perform this inference and find cuts with sparsity controlled by a parameter $\rho$:
\begin{alignat}{2}
    \max_{\alpha, \beta, z, \eta} \quad & \frac{1}{|\I|} \sum_{i \in \I} z_i - \rho \left(\|\alpha\|_0 + \|\beta\|_0\right) + \epsilon \eta \label{eq:cut} \tag{$\mathcal{P}_\text{cut}$}\\
    \text{s.t.} \quad & a_i^\top x_i^\star - b_i \le 0, &\qquad& \forall i \in \I, \nonumber \\
    & a_i^\top \hat{x}_i - b_i \ge \eta - M (1 - z_i), && \forall i \in \I, \nonumber \\
    & a_{i,k} = \alpha_{t(k)}^\top \phi_k(d_i), &\qquad& \forall i \in \I, \; \forall k \in \mathcal{K}(d_i), \nonumber \\
    & b_i = \beta^\top \phi_0(d_i), &\qquad& \forall i \in \I, \nonumber \\
    & z \in \{0,1\}^{|\I|}, \quad \eta \ge \underline{\eta}, \quad \|\alpha\|_\infty \le 1, \quad \|\beta\|_\infty \le 1. \nonumber
\end{alignat}
The first constraint makes the cut valid for every target, and the binary $z_i$ certifies that $\hat{x}_i$ violates it by at least the margin $\eta \ge \underline{\eta} > 0$. The objective maximizes the share of solutions of $\M$ that are cut off, penalizes the number of nonzero weights $\|\alpha\|_0 + \|\beta\|_0$ (linearized with binaries $u_{i, k} \ge |\alpha_{i, k}|$ and $v_{i} \ge |\beta_i|$, which also yield the bound $\|\alpha\|_\infty \le 1$ and $\|\beta\|_\infty \le 1$) and breaks ties with a small reward on the margin with $0 < \epsilon \ll 1$. The penalty $\rho$ trades separation for sparsity: a $\rho$ close to 0 yields the densest and most separating cut, whereas a large $\rho$ yields cuts with few terms, which are more interpretable and less prone to overfitting. The cut search sweeps $\rho$ downward to collect cuts of increasing density, from the sparsest nonempty one to the unregularized one, and evaluates each by the separation loss.

\subsection{Template completion} \label{app:template_completion}

We calibrate the template weights $w$ to preserve all observed targets while maximally separating the optimal solutions of the current formulation $\M$. Preservation requires each $x_i^\star$ to admit a lifting $y_i^\star$ into the non-primary variables of $\M$ and an assignment of the new variables $z_i^\star$ such that $(y_i^\star,z_i^\star)$ belongs to the new feasible set $\X_\tau(d_i; w)$. 
For the second goal, define \[ V_\tau(y,z,w;d) \triangleq \sum_{j\in J_\tau(d)} [g_j(y,z,w;d)]^+, \] which measures the total violation of the template constraints by a solution $(y,z)$. Since an optimal solution $\hat y_i$ of $\M$ does not specify the new variables $z$, it is excluded by the template only if \emph{every} completion $\hat z_i$ violates its constraints. Thus, for a given $w$, its violation is measured by $ \min_{\hat z_i \in Z_\tau(d_i)} V_\tau(\hat y_i,\hat z_i,w;d_i), $ i.e., the violation under the completion of the new variables that best satisfies the template.

Combining these two requirements, we determine the template weights $w$ by  \begin{alignat}{2} \max_{w \in \mathcal{W},\,(y_i^\star,z_i^\star)_{i\in\I}} \quad & \sum_{i\in\I} \min_{\hat z_i\in Z_\tau(d_i)} V_\tau(\hat y_i,\hat z_i,w;d_i) \label{eq:constraint-calibration} \tag{$\mathcal{P}_{\mathrm{constr}}$}\\ \text{s.t.} \quad & (y_i^\star,z_i^\star)\in \X_\tau(d_i;w), && \forall i\in\I, \label{eq:CC_feasible_target_lift}\\ & \Pi_0(y_i^\star)=x_i^\star, && \forall i\in\I. \label{eq:CC_is_lift} \end{alignat} 
For scale-invariant templates, we normalize the corresponding components of $w$. Problem~\eqref{eq:constraint-calibration} has a max-min structure, jointly calibrating the template parameters $w$ against the best completions $\hat z_i$ of the new variables. We solve this problem via constraint generation and discard any template where no weights allow every target to be feasibly lifted.

\subsection{Objective Calibration} \label{app:obj_calibration}

As with cuts, we parameterize objective coefficients as affine functions of instance data, with one weight vector per variable family. The objective coefficient of variable $k$ is $c_\theta(d_i)_k = \theta_{t(k)}^\top \phi_k(d_i)$, where $\phi_k(d_i)$ denotes data associated with variable $k$ and $\theta_{t(k)}$ is the vector to be learned for family $t(k)$.

We seek weights $\theta$ that balance two goals. First, each target solution $x_i^\star$ should admit a feasible lifting $y_i^\star$ that is nearly optimal under the calibrated objective. We allow some suboptimality to accommodate potentially suboptimal target solutions in practice. Second, the lifted target $y_i^\star$ should outperform the current model solution $\hat y_i$ by as large a margin as possible; equivalently, we seek an objective vector $c_\theta(d_i)$ for which $c_\theta(d_i)^\top(\hat y_i-y_i^\star)$ is large. We therefore calibrate $\theta$ by solving 
\begin{alignat}{2} \min_{\theta \in \Theta,\, (y_i^\star,\e_i)_{i \in \I}} \quad & \sum_{i \in \I} \e_i - \lambda \sum_{i \in \I} c_\theta(d_i)^\top(\hat y_i-y_i^\star) \label{eq:objective-calibration} \tag{$\mathcal{P}_{\mathrm{obj}}$} \\ \text{s.t.}\quad & c_\theta(d_i)^\top(y_i^\star-y) \leq \e_i && \forall i \in \I,\ \forall y \in \X(d_i), \label{eq:OC-subopt} \\ & y_i^\star \in \X(d_i), \qquad \Pi_0(y_i^\star)=x_i^\star && \forall i \in \I, \\ & \e_i \geq 0 && \forall i \in \I . 
\end{alignat}
At an optimum, $\e_i = c_\theta(d_i)^\top y_i^\star - \min_{y\in\X(d_i)} c_\theta(d_i)^\top y \geq 0,$ so $\epsilon_i$ precisely captures the suboptimality of the lifted target solution under the calibrated objective. The parameter $\lambda>0$ controls the tradeoff between target suboptimality and separation from the current model solutions. As with template calibration, we solve Problem~\eqref{eq:objective-calibration} via constraint generation.

\end{document}